\documentclass[reqno, 12pt]{article}
\usepackage{graphicx} % Required for inserting images
\usepackage{authblk}
\usepackage{amsthm}
\usepackage{amsfonts}
\usepackage{amssymb}
\usepackage{amsmath}
\usepackage{booktabs}
\usepackage{multirow}
\usepackage{multicol}
\usepackage{longtable}
\usepackage[style=alphabetic, citestyle=alphabetic, maxcitenames=10, maxbibnames=99]{biblatex}
\usepackage[margin=1in]{geometry}
\newtheorem{theorem}{Theorem}
\newtheorem{definition}{Definition}
\newtheorem{lemma}{Lemma}

\newtheorem{proposition}{Proposition}
\newtheorem*{proposition*}{Proposition}
\newtheorem{assumption}{Assumption}

\usepackage[ruled]{algorithm2e}
\SetKwInput{KwInput}{Input}
\SetKwInput{KwOutput}{Output}
\usepackage{hyperref}
\usepackage{setspace}
\usepackage{lineno}
\usepackage{xcolor}
\title{Learning collective variables that preserve transition rates}
\author[1]{Shashank Sule \thanks{ssule25@umd.edu}}
\author[2]{Arnav Mehta \thanks{arnavmehta@berkeley.edu}}
\author[1]{Maria K. Cameron \thanks{mariakc@umd.edu}}
\affil[1]{\small{Department of Mathematics, University of Maryland, College Park, MD, USA}}
\affil[2]{\small{Department of Mathematics, University of California, Berkeley, CA, USA}}

\date{\today}

\begin{document}
\maketitle
\begin{abstract}
    Collective variables (CVs) play a crucial role in capturing rare events in high-dimensional systems, motivating the continual search for principled approaches to their design. In this work, we revisit the framework of quantitative coarse graining and identify the orthogonality condition from Legoll and Lelievre (2010) as a key criterion for constructing CVs that accurately preserve the statistical properties of the original process. We establish that satisfaction of the orthogonality condition enables error estimates for both relative entropy and pathwise distance to scale proportionally with the degree of scale separation. Building on this foundation, we introduce a general numerical method for designing neural network-based CVs that integrates tools from manifold learning with group-invariant featurization. To demonstrate the efficacy of our approach, we construct CVs for butane and achieve a CV that reproduces the anti-gauche transition rate with less than ten percent relative error. Additionally, we provide empirical evidence challenging the necessity of uniform positive definiteness in diffusion tensors for transition rate reproduction and highlight the critical role of light atoms in CV design for molecular dynamics. 
\end{abstract}
%\tableofcontents
\section{Introduction}
\label{sec: introduction}
Collective variable (CV) discovery is a fundamental problem in computational statistical physics involving the design of coarse-grained low-dimensional representations of high-dimensional stochastic processes. In practice, CV discovery is particularly important for the design of novel materials \cite{neha2022collective, zhang2019improving, karmakar2021collective} or drug discovery \cite{cuchillo2015collective, alonso2006combining} and for the understanding of rare events in molecular dynamics (MD) simulations of biomolecules and chemical reactions \cite{yang2019enhanced, paul2019phase, valsson2016enhancing}. A standard model for such MD simulations is the \emph{overdamped Langevin dynamics}, which describes the motion of $N$ coordinates in a potential energy landscape: 
\begin{align}
    dX_t = -\nabla V(X_t)\,dt + \sqrt{2\beta^{-1}}\,dW_t. \label{eq: OLD}
\end{align}
Here $X_t \in \mathbb{R}^{N}$ represents the system's state, $V: \mathbb{R}^{N} \to \mathbb{R}$ is the potential energy function, $\beta^{-1} = k_B T$ where $T$ is the temperature and $k_B$ the Boltzmann constant, and \( W_t \) is a standard Wiener process. The wells of the potential $V$ encode metastable states—regions where the system remains trapped for long periods before transitioning due to thermal fluctuations. Simulating and analyzing these transitions is computationally expensive due to the wide range of timescales involved. Therefore, instead of working directly with the high-dimensional process $X_t$, we seek a collective variable $\xi: \mathbb{R}^{N} \to \mathbb{R}^d$ that maps $X_t$ to a low-dimensional process  $Y_t = \xi(X_t)$. Typically, CVs are designed with the following objectives:  
\begin{enumerate}
    \item Mapping by $\xi$ should separate metastable states, enabling a clear description of transitions,  
    \item The CV should enable enhanced sampling to accelerate rare-event simulations, and  
    \item The mapped process $Y_t = \xi(X_t)$ should reproduce the statistical properties of the original dynamics.    
\end{enumerate}
Thus, while traditional feature learning techniques in artificial intelligence optimize for expressivity or discriminative power, CVs must respect the dynamical and geometric properties of $X_t$, making their construction significantly more challenging. There is nonetheless a rich line of work for designing metastable state-separating CVs and CVs for enhanced sampling using data-driven methods, thus satisfying objectives (1) and (2). However, a key hindrance to achieving the objective (3) is that the projected process $Y_t$ is generally not closed because its evolution depends not only on its current state but also on latent high-dimensional variables. This lack of closure makes direct simulation and statistical reproduction of $Y_t$ infeasible. This is seen by applying Ito's lemma to $\xi(X_t)$: 
\begin{align}
    dY_t = d\xi(X_t) = \mathcal{L}\xi(X_t)\,dt + \sqrt{2\beta^{-1}}\textsf{det}(D\xi(X_t)D\xi(X_t)^{\top})^{1/2} \:dB_t. \label{eq: projected dynamics}
\end{align}
Here $B_t$ is a rescaled Brownian motion in $\mathbb{R}^d$ and $\mathcal{L}$ is the infinitesimal generator of the overdamped Langevin dynamics \eqref{eq: OLD}:
\begin{align}
    B_t &= \frac{D\xi}{\textsf{det}\left(D\xi D\xi^{\top}\right)^{1/2}}W_t, \label{eq: conditioning formulae} \\
    \mathcal{L}f &:= \beta^{-1}\Delta f - \nabla V \cdot \nabla f.     \label{eq:generator}
\end{align}

One approach to constructing a closed effective process is the \emph{Zwanzig closure}, which ensures that the time-marginals of the reduced process match those of the full dynamics \parencite{givon2004extracting}. However, this results in a non-Markovian equation, which is difficult to simulate exactly. Legoll and Lelièvre \parencite{legoll2010effective} studied an alternative \emph{effective dynamics} $Z_t$, where instead of conditioning on the full history of $X_t$, the evolution of \( Y_t \) is conditioned on the invariant measure $\mu = Z^{-1} \exp{(-\beta V})$ associated with $X_t$: 

\begin{align}
    dZ_t = b(Z_t)\,dt + \sqrt{2\beta^{-1}}M^{1/2}(Z_t)\,dB_t. \label{eq: effective dynamics}
\end{align}
The drift $b$, the matrix $M$ and the process $B_t$ are described as follows:
\begin{align}
\label{eq: effective_b}
    b(z) &= \mathbb{E}_{\mu}[\mathcal{L}\xi(X) \mid \xi(X) = z], \\
\label{eq: effective_A}
    M(z) &= \mathbb{E}_{\mu}[D\xi(X)D\xi(X)^{\top} \mid \xi(X) = z].
\end{align}
Thus, the effective dynamics $Z_t \in \mathbb{R}^d$ are constructed by conditioning the drift and diffusion of the projected dynamics \eqref{eq: projected dynamics} $Y_t$ by fixing $\xi$ under the measure $\mu$. $Z_t$ is intrinsically low-dimensional and can be simulated using established numerical methods \cite{weinan2004metastability, maragliano2006string}. The downside of the effective dynamics is that the laws of $Z_t$ and $Y_t$ do not coincide, so $Z_t$ may obfuscate the statistics of $Y_t$. Finding error estimates for guaranteeing the proximity of $Z_t$ and $Y_t$ is therefore a highly active topic of research with a number of works \cite{legoll2010effective, legoll2019effective, lelievre2019pathwise, duong2018quantification, Sharma2017} exploring the cases arising from different combinations of assumptions on the underlying dynamics $X_t$, the CV $\xi$, and ways to measure distances between stochastic processes. In this paper we are concerned with the \emph{scale separating} case of the underlying dynamics where the potential $V:= V_\epsilon$ is given by a sum of a \emph{confining potential} $V_1\ge 0$ and a \emph{driving potential} $V_0$.
The confining potential is
defined so that the level set 
\begin{align}
 \mathcal{M} := \{V_1 = 0\}  \label{eq: residence manifold}
\end{align}
has co-dimension 1 in $\mathbb{R}^{N}$. The effect of the confining potential is controlled by the small parameter $\epsilon \ll 1$: 
\begin{align}
    V = V_0 + \epsilon^{-1}V_1. \label{eq: split potential} 
\end{align}
% The confining potential %is defined by squaring the \emph{confining term} $V_1$ in order to maintain non-negativity of $V_\epsilon$. It has been 
% defined so} that the level set 
% \begin{align}
%  \mathcal{M} := \{V_1 = 0\}  \label{eq: residence manifold}
% \end{align}
%has co-dimension 1 in $\mathbb{R}^{N}$. 
Potentials of the form \eqref{eq: split potential} arise routinely in MD simulations where the inter-atomic energy is given by a sum of strong covalent bonds (corresponding to the stiff $\epsilon^{-1}V_{1}$ term) and weak torsion energies (corresponding to the term $V_0$). In this scale-separating case, the small parameter $\epsilon$ confines $X_t$ to reside near $\mathcal{M}$, and is henceforth termed the \emph{residence manifold}. On $\mathcal{M}$, the underlying slow process is directed by the driving potential $V_0$. More notably, it is shown in \cite{legoll2010effective} that the distance between $Z_t$ and $Y_t$ as measured by relative entropy can be made to scale as $C + O(\epsilon)$ if the following equation, the \emph{orthogonality condition}, is satisfied:   
\begin{align}
    D\xi\nabla V_1 = 0. \label{eq: orthogonality condition} \tag{OC}
\end{align}
In a toy two-dimensional example \parencite{legoll2010effective} demonstrated that designing CVs which satisfied \eqref{eq: orthogonality condition} was instrumental in reproducing residence times via the effective dynamics $Z_t$. However, most realistic MD simulations are high-dimensional and consist of not one, but several confining potentials at different scales, resulting in complicated residence manifolds of intrinsic dimension significantly smaller than $N$. 

% In this paper, we propose a numerical method which extends \parencite{legoll2010effective}'s example to obtain CVs for high-dimensional MD data without requiring knowledge of $V_1$ beforehand.  

\subsection{Aims and contributions} 
In this paper, we aim to apply the theory of quantitative coarse graining as outlined in \cite{legoll2010effective, duong2018quantification, lelievre2019pathwise} towards a numerical method for obtaining CVs. 

% We also propose a numerical method which extends \parencite{legoll2010effective}'s example to obtain CVs for high-dimensional MD data without requiring knowledge of $V_1$ beforehand. Our contributions are as follows: 

\begin{enumerate}
    \item We review the theory of quantitative coarse graining from \cite{legoll2010effective, duong2018quantification, lelievre2019pathwise} and demonstrate how two geometric conditions, the orthogonality condition \eqref{eq: orthogonality condition} and the projected orthogonality condition \eqref{eq: orthogonality condition 2}, emerge as sufficient conditions for reducing the discrepancy between the projected dynamics $Y_t$ \eqref{eq: projected dynamics} and the effective dynamics $Z_t$ \eqref{eq: effective dynamics} as quantified by relative entropy \eqref{eq: relative entropy estimate} and pathwise distance respectively \eqref{eq: pathwise distance}. We evaluate such analytical conditions on CVs in the context of reproduction of reaction rates in molecular dynamics simulations. In particular, we prove that \eqref{eq: orthogonality condition} and \eqref{eq: orthogonality condition 2} are equivalent (Proposition \eqref{prop: oc poc equivalence}). Moreover, we provide evidence that the reaction rate can be reproduced even if the diffusion tensor $M(z)$ is rank-deficient. We illustrate this fact for the anti-gauche transition rate in the \emph{normal}-butane molecule $C_{4}H_{10}$ (henceforth simply referred to as \emph{butane}) (Table \eqref{tab: transition rates}). 
    \item We propose a general procedure for designing CVs based on \eqref{eq: orthogonality condition} (Algorithm \eqref{alg: learning cv}). We show how to integrate several tools from manifold learning such as diffusion maps \cite{coifman2008diffusion}, diffusion nets \cite{mishne2019diffusion}, group invariance, and independent eigencoordinate selection \cite{chen2019nonlinear} to learn a neural network-based representation of $\mathcal{M}$ as \emph{hypersurface} (Algorithm \eqref{alg: ResManLearn}). We remark that this hypersurface $\widehat{\mathcal{M}}$, henceforth termed the \emph{surrogate manifold} may contain self-intersections despite the use of functionally independent coordinates. As a solution, we propose a novel manifold learning technique, Laplacian Conformal Autoencoder (LAPCAE) which can \emph{undo} these self-intersections (Figure \eqref{fig: butane slice}).  
        \item We underline the instrumental role played by group invariance in the post-processing of MD data. In particular, we propose mapping, or \emph{featurizing} data sampled from overdamped Langevin dynamics $X_t$ \eqref{eq: OLD} through a feature map $\mathcal{F}$ with invariances to rotations, translations, or reflections. We show (Figure \ref{fig: butane man learning}) that different feature maps can provide markedly different embeddings of the same data. Moreover, we demonstrate for butane that the positions of the hydrogen atoms are \emph{necessary} for learning accurate representations of the dynamics. This point is of particular interest to the development of MD workflows, where currently it is standard practice to post-process simulation data by discarding coordinates of light atoms. 
        \item We conduct a detailed investigation of our algorithm on butane as a case study. Although butane is a toy model for studying conformational dynamics, it is nonetheless the simplest MD example that exhibits the characteristics of MD data, such as rotation and translation invariance, low intrinsic dimensionality, metastability, and slow and fast time scales. Moreover, the existence of a clear CV, the dihedral angle, makes it easier to evaluate and benchmark our proposed three-step algorithm. As a lookahead to our results, the CV learned by our algorithm for the feature map {\sf PlaneAlign} reproduces the anti-gauche transition rate in butane %\emph{up to four digits of accuracy} 
       with a relative error of nearly $10\%$ while the dihedral angle yields a relative error of 24\%
        (Table \eqref{tab: transition rates}). The code to reproduce our experiments and train our models has been provided at \url{https://github.com/ShashankSule/CV_learning_butane}. 
    \end{enumerate}

\subsection{Related work} CV discovery is a thriving interdisciplinary subject at the intersection of probability theory, numerical analysis, biophysics, and, more recently, machine learning. On the theoretical front, since \citeauthor{legoll2010effective}'s foundational paper proposing the effective dynamics \eqref{eq: effective dynamics}, there have been several subsequent theoretical works that have extended the conditioning approach to multidimensional CVs, non-reversible processes, and non-isotropic diffusions. Moreover, these works have addressed distances between $Z_t$ and $Y_t$ as measured by relative entropy, Wasserstein distance, and the pathwise distance \parencite{duong2018quantification, Lelièvre2019, hartmann2020coarse, hartmann2020coarse, legoll2019effective}. On the computational side, there is a vast literature on designing CVs for MD simulations, starting with human-designed CVs \cite{heilmann2020sampling, ooka2023accurate, chodera2007automatic, hinsen2020hydrogen, altis2008dihedral, schwantes2013improvements}, data-driven CVs using classical unsupervised learning methods such as diffusion maps \cite{rohrdanz2011determination, zheng2013rapid, preto2014fast}, PCA (and variants) \cite{romero2007essential, grossfield2007application}, LDA \cite{sasmal2023reaction}, IsoMap \cite{spiwok2011metadynamics, zheng2013molecular}, and tICA \cite{m2017tica, schwantes2015modeling, schultze2021time} or its generalization, the variational approach for conformational (VAC) dynamics \cite{perez2013identification, noe2013variational, nuske2014variational}. With the burgeoning ability to train highly nonlinear deep models, deep learning methods for generalizing the supervised learning methods mentioned above such as VAMPnets \cite{chen2019selecting, mardt2018vampnets, bonati2021deep} and informational bottlenecked techniques such as time-lagged autoencoders \cite{wehmeyer2018time, zhang2023understanding, chen2018molecular} or variational autoencoders \cite{wang2021state, wang2024information, ribeiro2018reweighted, hernandez2018variational} have emerged. The present paper sits at the intersection of theory and application in the sense that it realizes the quantitative coarse-graining theory into an algorithm applied to an MD system. A closely related work towards the same goal is the recent paper \cite{zhang2024finding}.

\subsection{Organization} The rest of this paper is organized as follows. In Section \eqref{sec: background}, we review the theory of quantitative coarse-graining from \cite{legoll2010effective, duong2018quantification, lelievre2019pathwise} where we state relative entropy and pathwise distance estimates for the discrepancy between the projected dynamics and the effective dynamics. In Section \eqref{sec: effective dynamics analysis}, we study two conditions emerging from this theory, the orthogonality condition \eqref{eq: orthogonality condition} and the uniform positive definiteness of $D\xi$ \eqref{ass: xi}. We also propose a numerical method for designing neural network-based CVs. The following two sections are more experimental and concern the application of our algorithm to butane. In particular, we describe how to learn the surrogate manifold $\mathcal{\widehat{M}}$ in Section \ref{sec: resmanlearn} and we learn the CV and present results for transition rate estimation in Section \ref{sec: learning cvs}. 

\section{A brief overview of quantitative coarse-graining}
\label{sec: background}

\subsection{Collective Variables} 
In this section, we review results from \cite{legoll2010effective, duong2018quantification, Lelièvre2019} to describe how \eqref{eq: orthogonality condition} emerges as a means to optimize coarse-graining error between $Z_t$ \eqref{eq: effective dynamics} and $Y_t$ \eqref{eq: projected dynamics}. To begin, we state some general regularity conditions which ensure the existence of the invariant measure $\mu$ and unique solutions to the all-atom dynamics \eqref{eq: OLD} and the effective dynamics \eqref{eq: effective dynamics}:
\begin{assumption}
\label{ass: existence and uniqueness}
    We assume the potential $V$ satisfies: 
    \begin{align}
        &V \in C^{\infty}(\mathbb{R}^{N}) \text{ almost everywhere}, \quad \exp{(-\beta V)} \in L^{1}(\mathbb{R}^{N}). \label{eq: integrability of V}
    \end{align}
    Moreover, we assume that the effective drift \eqref{eq: effective_b} and the effective diffusion \eqref{eq: effective_A} are Lipschitz: 
    \begin{align}
         \exists \: L_M, L_b > 0 \text{ such that }|b(z) - b(z')| \leq L_b|z - z'|, \quad \|M(z) - M(z')\|_F \leq L_M|z - z'|. \label{eq: Lipschitz constants}
    \end{align}
\end{assumption}
Conditions \eqref{eq: integrability of V} and \eqref{eq: Lipschitz constants} ensure that both $X_t$ and $Z_t$ exist and are unique. We note here that \cite{duong2018quantification} gives a more technical assumption for the CV $\xi$ (see \cite[Eq. (C3)]{duong2018quantification}). However, insofar as we are concerned with relative entropy estimates, this assumption is only used for proving that the effective dynamics have Lipschitz coefficients. Therefore, by assuming the above Lipschitzness we can remove this more technical assumption appearing in \cite{duong2018quantification}. A more consequential condition for studying the effective dynamics is the boundedness of the singular values of $D\xi D\xi^{\top}$:
\begin{assumption}
\label{ass: xi}
We assume that $\xi \in C^{3}(\mathbb{R}^{N})$. Moreover, there exists $\delta_{\xi} > 0$ such that the collective variable $\xi$ satisfies: 
    \begin{align}
        D\xi D\xi^{\top} \succcurlyeq \delta^{2}_{\xi} I_{d}
    \end{align}
    Here $I_d$ is a $d \times d$ identity matrix.
\end{assumption}

Given Assumption \ref{ass: xi}, the conditioning approach may be further clarified through the use of an integral formulation. First, we introduce the level set $\Sigma_{z} := \{\xi = z\}$. Since assumption \eqref{ass: xi} ensures that $z$ is a regular value of $\xi$, by the implicit function theorem, $\Sigma_z$ is a smooth manifold of codimension $d$ whose cotangent space is spanned by the rows of the Jacobian $D\xi$. In particular, the level sets $\left(\Sigma_z\right)_{z \in \mathbb{R}^{d}}$ \emph{foliate} $\mathbb{R}^{N}$ via:
\begin{align}
    \mathbb{R}^{N} = \coprod_{z \in \mathbb{R}^d} \Sigma_{z}. 
\end{align}
Consequently, the \emph{co-area formula}~\cite{Evans1992} can be used to break up the integral on $\mathbb{R}^{N}$ into integrals on $\Sigma_z$. For instance, if $\varphi: \mathbb{R}^{N} \to \mathbb{R}$ is bounded and continuous then 
\begin{align}
    \int_{\mathbb{R}^{N}}\varphi(x)\mu(x)\,dx = \int_{\mathbb{R}^{d}}\int_{\Sigma_{z}}\frac{\varphi(x)\mu(x)}{\textsf{det}(D\xi D\xi^{\top})}d\sigma_{\Sigma_z}(x)\,dz. \label{eq: coarea formula}
\end{align}
Note that if $X \sim \mu$ is a random variable, then so is $Z = \xi(X)$. Then \eqref{eq: coarea formula} is a geometric description of the tower law
\begin{align}
    \mathbb{E}_{\mu}[\varphi(X)] = \mathbb{E}\left[\mathbb{E}\left[\varphi(X) \mid Z\right]\right]. \label{eq: tower law}
\end{align}
The manifold $\Sigma_z$ can then be used to describe the probability densities of random variables after conditioning on the values of $\xi$. In particular, if $X$ is distributed by the invariant measure $\mu=Z^{-1}\exp(-\beta V(x))$ then $\xi(X) \in \mathbb{R}^d$ is distributed according to the measure $\mu^{\xi}$ given by
\begin{align}
    \mu^{\xi}(z) dz = Z^{-1}\int_{\Sigma_{z}}\frac{\exp \left(-\beta V(x)\right)}{\sqrt{\textsf{det}(D\xi D\xi^{\top}(x))}}\,d\sigma_{\Sigma_z}(x)\,dz.
    \label{eq: conditional measure}
\end{align}
Therefore, we can define a new measure on the surface $\Sigma_z$ given by $\mu_z$: 
\begin{align}
    \mu_z(x)\,dx &:= N_{z}^{-1}\frac{Z^{-1}\exp \left(-\beta V(x)\right)}{\sqrt{\textsf{det}(D\xi D\xi^{\top}(x))}}\,d\sigma_{\Sigma_z}(x),\\ 
    N_z &= Z^{-1}\int_{\Sigma_{z}}\frac{\exp \left(-\beta V(x)\right)}{\sqrt{\textsf{det}(D\xi D\xi^{\top}(x))}}\,d\sigma_{\Sigma_z}(x). 
\end{align}
The normalization constant $N_z$ is used to define the \emph{free energy}: 
\begin{align}
    f(z) = -\beta^{-1}\log N_z.  \label{eq: free energy}
\end{align}
By expressing the free energy as an integral we can also write the derivative of the free energy as an integral. This is formulated in \parencite[Lemma 2.2]{legoll2010effective} and \parencite[Lemma 2.4]{duong2018quantification}: 

\begin{lemma}
\label{lem: mean force lemma}
    The derivative of the free energy $\nabla_z f(z)$, termed the \emph{mean force} is given by
    \begin{align}
        \nabla_z f(z) &= \int_{\Sigma_z} F_z(x)\,dx, \\
        F_z(x) &= (D\xi D\xi)^{-1}D\xi \nabla V - \beta^{-1}\nabla \cdot ((D\xi D\xi^{\top})^{-1} D\xi). \label{eq: local mean force} 
    \end{align}
    Here $\nabla \cdot$ is the divergence operator and is applied row-wise when applied to a matrix. 
\end{lemma}

Lemma \ref{lem: mean force lemma} reveals that the mean force can itself be represented as an integral where the integrand $F$ is the \emph{local mean force}. As we shall see, the local mean force will play a crucial role in designing a strategy for choosing a collective variable. 

% But first, we state an equivalence between the dynamics \eqref{eq: effective dynamics} and a reversible dynamics proposed by \parencite{maragliano2006string}. 

% \begin{proposition}
%     Consider the following \emph{free-energy} dynamics proposed by \parencite{maragliano2006string}: 
%     \begin{align}
%         \widehat{Z}_t = -(A(\widehat{Z}_t)\nabla_{z}f(\widehat{Z}_t) + \beta^{-1}\nabla \cdot A(\widehat{Z}_t))\,dt + \sqrt{2\beta^{-1}}A(\widehat{Z}_t)^{1/2}\,dW_t. \label{eq: free energy dynamics}
%     \end{align}
%     Then under the assumption \eqref{ass: xi}, the effective dynamics \eqref{eq: effective dynamics} $Z_t$ are almost surely identical to the free energy dynamics $\widehat{Z}_t$. 
% \end{proposition}

\subsection{Measuring the quality of effective dynamics}

Many questions around the quality of the CV $\xi$ boil down to measuring how close the effective dynamics process $Z_t$ is to the coarse-grained process $Y_t$. As such, there are many ways to quantify the proximity between two stochastic processes, but two particularly fruitful criteria, \emph{relative entropy} and \emph{pathwise convergence}, were given in \parencite{legoll2010effective, duong2018quantification, lelievre2019pathwise}. We first discuss the relative entropy criterion.

\subsubsection{Relative entropy}

\begin{definition}[Relative entropy]
    Let $\nu, \eta$ be two probability measures on $\mathbb{R}^d$ which are absolutely continuous with respect to the Lebesgue measure. Then the relative entropy $H(\nu,\eta)$ is given as follows: 
    \begin{equation}
        H(\nu,\eta) := \begin{cases}
            \int_{\mathbb{R}^d} \, \ln\left(\frac{d\nu}{d\eta}\right)\,d\nu, \quad &\text{if} \quad \nu \ll \eta, \\
            +\infty \quad &\text{otherwise}.
        \end{cases}
    \end{equation}
\end{definition}

The relative entropy is related to measuring distances between distributions due to the Csiszar-Kullback inequality:

\begin{proposition}[Cziszar-Kullback inequality]
    Let $\|\nu - \eta\|_{TV}$ be the \emph{total variation} distance between the probability measures $\nu$ and $\eta$ on the $\sigma$-algebra $\mathcal{G}$. Then 
    \begin{align}
        \|\nu - \eta\|_{TV} := \underset{A \in \mathcal{G}}{\sup}|\mu(A) - \nu(A)| \leq \sqrt{2 H(\nu,\eta)} \label{eq: cziszar-kullback}.
    \end{align}
\end{proposition}

The Cziszar-Kullback inequality shows that reducing relative entropy between distributions is a bona fide way of making two distributions close to each other in the TV norm. This was the content of the main result  \parencite{legoll2010effective} which showed how to reduce the relative entropy between the $\mathcal{P}(Y_t)$ and $\mathcal{P}(Z_t)$ distributions on $\mathbb{R}^d$ induced by $Y_t$ and $Z_t$. In \parencite{duong2018quantification}, this result was generalized for a $d$-dimensional collective variable where \cite{duong2018quantification} imposed a few additional assumptions: 

\begin{assumption}
\label{ass: rho}
    There exists $\rho > 0$ such that for every probability measure $\nu$ satisfying $\nu \ll \mu_{z}$
    \begin{align}
        H(\nu, \mu_{z}) \leq \frac{1}{2\rho}\int_{\Sigma_{z}}\left\|\nabla \ln \left( \frac{d\nu}{d\mu_z}\right)\right\|^2\,d\nu. \label{eq: log sobolev}
    \end{align}
    The above inequality is the \emph{log-Sobolev} inequality and $\rho$ is the log-Sobolev constant. 
\end{assumption}

The log-Sobolev constant determines how quickly a Markov chain mixes to its invariant distribution. In particular, if $\partial_{t}\mu_t = \mathcal{L}^{\dagger}\mu_t$ is the Fokker-Planck equation and $\mathcal{L}^{\dagger}\mu = 0$ is the stationary solution, then by Boltzmann's H-theorem (see e.g. \cite[Lemma 2]{vempala2019rapid}), we have 
\begin{align}
    H(\mu_t, \mu) \leq \exp(-2\rho t)H(\mu_0, \mu). \label{eq: boltzmann h-theorem}
\end{align}
If $\rho$ is large, then the associated diffusion process on $\mu_z$ mixes quickly. Since Assumption \eqref{ass: rho} stipulates that every $\mu_z \in \mathbb{R}^d$ shares a single $\rho$, this implies that processes on $\Sigma_z$ mix \emph{uniformly quickly}. Next, we turn to an assumption regarding the local mean force $F_z$ in \eqref{eq: effective dynamics}. We can endow $\Sigma_z$ with a geodesic distance: 
\begin{align}
    d_{z}(x_1, x_2) := \inf\left\{\int_{0}^{1}\|\gamma'(t)\|\,dt \mid \gamma(0) = x_1, \gamma(1) = x_2, \gamma(t) \in \Sigma_{z}\right\}.
\end{align}

\begin{assumption}
\label{ass: kappa}
    We assume that the local mean force is $\kappa$-Lipschitz from $\Sigma_z$ to $\mathbb{R}^d$ as measured by the domain metric $d_z$ and the range metric $\|x-y\|_{M}^2 = (x-y)^{\top}M(x-y)$: 
    \begin{align}
        \kappa := \underset{z \in \mathbb{R}^d}{\textsf{sup}}\:\underset{x_1,x_2 \in \Sigma_z}{\textsf{sup}}\frac{\|F_z(x_1) - F_z(x_2)\|_{M(z)}}{d_{z}(x_1, x_2)} < +\infty, \label{eq: kappa}
    \end{align}
    where $F_z$ is the local mean force given by \eqref{eq: local mean force}.
\end{assumption}

The last assumption relates to the variation of $D\xi D\xi^{\top}$ from its mean $M(z)$ \eqref{eq: effective_A} on each level set $\Sigma_z$. 

\begin{assumption}
\label{ass: lambda}
    Since $D\xi$ is smooth, its restriction to $\Sigma_z$ is well-defined and is measurable. Let $L(z): \mathbb{R}^d \to \mathbb{R}$ be given by 
    \begin{align}
        L(z) := \|(M(z) - D\xi(x)D\xi(x)^{\top})(D\xi(x)D\xi(x)^{\top})^{-1/2}\|_{L^{\infty}(\Sigma_{z})}.
    \end{align}
    We assume that
    \begin{align}
         \lambda := \|L(z)\|_{L^{\infty}(\mathbb{R}^d)} < +\infty. 
    \end{align} 
\end{assumption}
The constant $\lambda$, in some sense, measures how far $\xi$ is from being an affine function. If, for instance, $\xi = Ax + b$ then $\lambda = 0$ because $D\xi = A$ and $M(z) = AA^\top$. 
% \begin{remark}
%     The definition of $\kappa$ can seem overtly technical but it was shown in \parencite{legoll2010effective} that reducing this constant is crucial in choosing a collective variable. We will illustrate this argument in the next section. 
% \end{remark}

\parencite{legoll2010effective} proved the following result which was later generalized by \parencite{duong2018quantification}:

\begin{theorem}
\label{thm: duong + legoll} Let $\mathcal{P}(Z_t), \mathcal{P}(Z_t)$ denote the distributions of $Z_t, Y_t$ and suppose that $Z_0$ and $Y_0$ are identically distributed. Under the Assumptions \ref{ass: existence and uniqueness}, \ref{ass: xi}, \ref{ass: rho}, \ref{ass: kappa}, and \ref{ass: lambda},
\begin{align}
    H(\mathcal{P}(Z_t), \mathcal{P}(Y_t)) \leq \frac{1}{4}\left(\lambda + \frac{\kappa^{2}\beta^{2}}{\rho^2}\right)\left(H(\mathcal{P}(X_0), \mu) - H(\mathcal{P}(X_t), \mu)\right). \label{eq: relative entropy estimate}
\end{align}
If $D\xi$ is constant on the level set $\Sigma_z$ then $\lambda = 0$. 
\end{theorem}

The estimate \eqref{eq: relative entropy estimate} was referred to as an ``intermediate time" estimate in \cite{legoll2010effective}. The reason is as follows: as $t \to \infty$, it was shown that the effective and projected dynamics have the same invariant measure given by $\xi_{\#}\mu$, the pushforward of $\mu$ by $\xi$. Thus, the left-hand side of \eqref{eq: relative entropy estimate} decays to zero while the right-hand side does not, so the bound is not sharp in the $t \to \infty$ limit. Moreover, if $t = 0$ then both sides of \eqref{eq: relative entropy estimate} vanish. Thus, the estimate \eqref{eq: relative entropy estimate} prescribes control on intermediate $t \in (0, \infty)$. 

\subsubsection{Reducing the relative entropy and the orthogonality condition}

Suppose $V$ satisfies the scale separation condition:
\begin{align}
    V = V_0 + \epsilon^{-1}V_1. \tag{\ref{eq: split potential}} 
\end{align}
In this case, \parencite{legoll2010effective} computed the local mean force in \eqref{eq: local mean force} explicitly \parencite[p. 2150]{legoll2010effective}: 
\begin{align}
    F = \frac{2}{\epsilon} \frac{D\xi \nabla V_1}{\|D\xi\|_{2}^{2}} + \frac{D\xi \nabla V_0}{\|D\xi\|_{2}^{2}} - \beta^{-1}\nabla \cdot  \left(\frac{D\xi^{\top}}{\|D \xi\|_{2}^{2}}\right).
\end{align}
From now on, we will omit the subscript $z$ from the local mean force for brevity.
Since $\epsilon \ll 1$, for a general CV $F = O(1/\epsilon)$ and therefore $\kappa$ \eqref{eq: kappa} will be $O(1/\epsilon)$. However, if \eqref{eq: orthogonality condition}: 
\begin{align*}
    D\xi \nabla V_1 = 0
\end{align*}
is satisfied then $F= O(1)$. Moreover, in this case, $\rho$ \eqref{eq: log sobolev} becomes $O(\epsilon)$ and therefore, the relative entropy estimate reduces to (see \cite[p. 2150]{legoll2010effective} or \cite[Eq. (4.3)]{duong2018quantification}):
\begin{align}
    H(\mathcal{P}(Z_t), \mathcal{P}(Y_t)) \leq \frac{1}{4}\left(\lambda + O(\epsilon)\right)\left(H(\mathcal{P}(X_0), \mu) - H(\mathcal{P}(X_t), \mu)\right). \label{eq: relative entropy estimate, order epsilon}
\end{align}

\subsection{\cite{legoll2010effective}'s example} The above calculation shows how \eqref{eq: orthogonality condition} emerges as a way to make the error estimate \eqref{eq: relative entropy estimate} small. However, it also imparts a geometric requirement to the CV by stipulating that its level sets $\Sigma_z$ must be orthogonal to those of the confining term $V_1$. This was illustrated in \cite[Table 1]{legoll2010effective} where a toy two-dimensional system with $V = (x^2-1)^2 + 100(x^2 + y - 1)^2$ was considered and the quality of the effective dynamics \eqref{eq: effective dynamics} for two CVs, $\xi_1(x,y) = x$ and $\xi_2(x,y) = x\exp(-2y)$, was evaluated. Notably, $\xi_2$ satisfies \eqref{eq: orthogonality condition} while $\xi_1$ does not. The effective dynamics under $\xi_2$ were able to reproduce residence times with a relative error less than 1\%, while the effective dynamics under $\xi_1$ reproduced residence times with 23\% relative error (see \cite[Table 1]{legoll2010effective}). This example suggests that imposing \eqref{eq: orthogonality condition} significantly improves the CV design. 

Notably, the approach of using \eqref{eq: orthogonality condition} does not touch the constant $\lambda$ appearing in the reduced error estimate \eqref{eq: relative entropy estimate, order epsilon}. In fact, $\lambda$ is seemingly the more significant contributor to the error since it makes the error estimate \eqref{eq: relative entropy estimate, order epsilon} behave as $O(1)$ irrespective of \eqref{eq: orthogonality condition}. It has been noted in \cite{legoll2010effective} and \cite{duong2018quantification} that $\lambda = 0$ if (i) $\xi$ is affine or, more generally, (ii) $D\xi$ is constant on $\Sigma_z$ for every $z$. Both conditions (i) and (ii) are rather restrictive. 
%since they exclude several important nonlinear functions. 
Therefore, there seems to be a geometric trade-off in terms of reducing $\lambda$ and $\kappa$ simultaneously. In this paper, guided by \cite{legoll2010effective}'s example, we will focus on \eqref{eq: orthogonality condition} exclusively, leaving this fascinating trade-off as the subject of future investigation. 

% We will refer to \eqref{eq: orthogonality condition} as the \emph{Orthogonality Condition \eqref{eq: orthogonality condition}}. This will be the criterion we will use for learning the collective variable $\xi$. OC also has a geometric interpretation: when the potential satisfies \eqref{eq: split potential}, the trajectories of $X_t$ are confined to be near the level set $\mathcal{M} = \{V_1 = 0\}$, the \emph{residence manifold} of the dynamics $X_t$. If $0$ is a regular value of $V_1$, $\nabla V_1$ is the normal vector to the manifold $\mathcal{M}$ so OC implies that the gradients of the collective variable must be orthogonal to the normal vector of $\mathcal{M}$, and thus must reside in the tangent space $T_p \mathcal{M}$ at every point $p \in \mathcal{M}$. This geometric interpretation allows the extension of OC to potentials that may not necessarily satisfy \eqref{eq: split potential}, because the process $X_t$ may nonetheless reside on a low-dimensional manifold $\mathcal{M}$. In this case, a way to design the collective variable would be to impose \eqref{eq: orthogonality condition} with the normal vector to the manifold in place of $\nabla V_1$. 

\subsubsection{Pathwise estimates}

Next, we discuss another criterion for measuring the distance between two stochastic processes: the \emph{pathwise distance} defined as follows: 

\begin{definition}
Let $A_t$, $B_t$ be two $\mathbb{R}^d$-valued stochastic processes defined on $t \in [0,T]$ and adapted to the same filtration $\left(\mathcal{F}_{t}\right)_{0 \leq t \leq T}$. Then the pathwise distance between $A_t$ and $B_t$ is given by
\begin{align}
    d(A_t, B_t) := \mathbb{E}\left[\underset{t \in [0,T]}{\textsf{sup}}\left|A_t - B_t\right|\right]. \label{eq: pathwise distance}
\end{align}
\end{definition}
Note that two processes with zero pathwise distance induce, for a fixed $t$, the same distribution on $\mathbb{R}^d$ and therefore have zero relative entropy. The works \parencite{legoll2010effective,legoll2017pathwise,lelievre2019pathwise} describe the conditions under which the projected process \eqref{eq: projected dynamics} and effective dynamics \eqref{eq: effective dynamics} may be brought close in terms of this much stronger distance between stochastic processes. To set up this result, we define the projection operator $\Pi$ induced by the collective variable: 
\begin{equation}
    % \Pi = I - \sum_{i,j} (\Phi^{-1})_{ij} \nabla \xi_i \otimes (\nabla \xi_j) \label{eq: Pi}
    \Pi = I - \sum_{i,j} (\Phi^{-1})_{ij} \nabla \xi_i\nabla \xi_j^\top\label{eq: Pi}
\end{equation}
where $\Phi = D\xi D\xi^{\top}$. With the projection operator $\Pi$, the infinitesimal generator~(\ref{eq:generator}) can be decomposed as 
\begin{equation}
    \mathcal{L} = \mathcal{L}_0 + \mathcal{L}_1
\end{equation}
where 
\begin{align}
    \mathcal{L}_0 & = \frac{e^{\beta V}}{ \beta} \sum_{i,j} \delta_{ij}\frac{\partial}{\partial x_i} \left( e^{-\beta V} \Pi_{ij} \frac{\partial}{\partial x_j}\right), \\
    \mathcal{L}_1 & = \frac{e^{\beta V}}{ \beta} \sum_{i,j} \delta_{ij}\frac{\partial}{\partial x_i} \left( e^{-\beta V} (I - \Pi)_{ij} \frac{\partial}{\partial x_j}\right). 
\end{align}

Furthermore, we assume
\begin{align}
    \kappa^2_1 &:= \sum_i \int (\Pi \nabla \mathcal{L} \xi_i) \cdot (\Pi \nabla \mathcal{L}\xi_i) \,d\mu < \infty, \notag\\
    \kappa^2_2 &:= \sum_i \int (\Pi \nabla M_{ij}) \cdot (\Pi \nabla M_{ij})d\mu < \infty.
\end{align} 
Next, we state the result of the pathwise distance between the effective dynamics the coarse-grained process below in Theorem~\ref{thm: pathwise distance theorem}: 
\begin{theorem}{\cite[Theorem 1]{lelievre2019pathwise}}
\label{thm: pathwise distance theorem}
    There exists a constant $\gamma(V, \xi) > 0$ such that:
    \begin{align}
        d(Y_t, Z_t) \leq \frac{3t}{\beta \gamma} \left(\frac{27 \kappa^2_1}{2 \gamma} + \frac{32 \kappa^2_2}{\beta} \right)e^{Lt} \label{eq: pathwise distance general}
    \end{align}
    where $L = 3L^2_b + \frac{48 L^2_M}{\beta} + 1$. Here $L_M, L_b$ are Lipschitz constants defined in Assumption \eqref{ass: existence and uniqueness}. 
\end{theorem} 

\subsubsection{Reducing the pathwise distance in the presence of scale separation.}

We consider the potential $V$ with the scale separation in the form of equation (\ref{eq: split potential}). In \parencite{lelievre2019pathwise} it was shown that the bound in \eqref{eq: pathwise distance general} may be made sharper (and in particular $O(\epsilon)$) if the collective variable $\xi$ is chosen to satisfy the \emph{projected orthogonality condition} ~\eqref{eq: orthogonality condition 2}
\begin{equation}
    (I - \Pi)^\top \nabla V_1 \equiv 0. \label{eq: orthogonality condition 2} \tag{POC}
\end{equation}
The \eqref{eq: orthogonality condition 2} concerns choosing the rows of the projection $(I-\Pi)$ to be orthogonal to $\nabla V_1$. Under \eqref{eq: orthogonality condition 2} and the scale separation condition \eqref{eq: split potential}, Theorem \ref{thm: pathwise distance theorem} can be applied to show that:  
\begin{equation}
    d(Y_t, Z_t) \leq t\left(\frac{C_1 \epsilon}{K} + \frac{C_2 \epsilon^2}{K^2} \right) e^{Lt} \sim O(\epsilon)
\end{equation}
Additionally, under \eqref{eq: orthogonality condition 2} we note 
\begin{align}
    \Pi^\top \nabla V = \Pi^\top \nabla V_0 + \frac{1}{\epsilon} \nabla V_1 = O(1/\epsilon), \quad (I - \Pi)^\top \nabla V = (I - \Pi)^\top \nabla V_0.
\end{align}

This implies that 
\begin{equation}
\label{eq: coefficient condition}
    \centering \textit{Operator $\mathcal{L}_0$ contains large coefficients, while $\mathcal{L}_1$ does not.}
\end{equation}
We term \eqref{eq: coefficient condition} the \emph{coefficient condition} (CC) and remark that \eqref{eq: orthogonality condition 2} is a special case of (CC). Moreover, in  \parencite{lelievre2019pathwise} it was noted that a good choice of collective variable in the most routinely occurring cases of scale separation falls under (CC). 

% \begin{remark}
% The theory of relative entropy summarized above is based on the works \parencite{duong2018quantification,legoll2010effective, legoll2017pathwise}, which concern all-atom dynamics modeled by \eqref{eq: OLD}. A more general analogue of Theorem \ref{thm: duong + legoll}, applicable to non-gradient and non-degenerate Ito diffusions, was proved in \parencite{hartmann2020coarse}. Moreover, our statement of Theorem \ref{thm: pathwise distance theorem} is a simplification of \parencite[Theorem 1]{lelievre2019pathwise} which applies to general reversible SDEs. 
% \end{remark}

\subsubsection{Simulating the effective dynamics} 
\label{subsec: simulating the effective dynamics}
CVs are essential for identifying and replicating the metastable dynamics of the all-atom dynamics \eqref{eq: OLD}. In particular, given $\xi$, we may identify metastable states $\tilde{A}, \tilde{B}$ typically corresponding to the wells of the free energy $f(z)$ \eqref{eq: free energy}. We then lift these metastable states to $\mathbb{R}^{N}$ by defining $A = \xi^{-1}(\tilde{A})$ and $B = \xi^{-1}(\tilde{B})$. A good CV should yield $A$ and $B$ such that the statistics of the rare transitions between $A$ and $B$ in all-atom space $\mathbb{R}^{N}$ are replicated by the effective dynamics in CV space $\mathbb{R}^{d}$. An important statistic is the transition rate $\nu_{AB}$ between $A$ and $B$ defined as follows:
\begin{equation}
\label{nuAB}
\nu_{AB} = \lim_{T\rightarrow\infty}\frac{N_{AB}}{T}.
\end{equation} 
Here $N_{AB}$ is the number of transitions made by $X_t$ from $A$ to $B$ observed during the time interval $[0,T]$. Using the framework of transition path theory (TPT)\cite{weinan2006towards}, the transition rate can be described via the committor function. In particular, the committor $q$ solves the \emph{committor PDE}
\begin{align}
\begin{cases}
   \mathcal{L}q(x) = 0, & x \in\Omega_{AB}:= \mathcal{M} \setminus (A\cup B), \\
   q(x) = 0, &x\in \partial A\\
    q(x) = 1,& x\in \partial B. 
   \end{cases}\label{eq: committor pde}
\end{align}
Then, using $q$, $\nu_{AB}$ may be expressed as: 
\begin{equation}
    \nu_{AB} = \beta^{-1} Z^{-1} \int_{\Omega_{AB} }\|\nabla q(x)\|^2\exp(-\beta V(x))dx. \label{eq: transition rate}
\end{equation}
Assuming that $\xi$ separates the metastable states $A$ and $B$, the effective dynamics may be used to numerically calculate $\nu_{AB}$ by finding the corresponding transition rate $\tilde{\nu}_{AB}$ for the low-dimensional effective dynamics. Below we describe a numerical method for computing $\tilde{\nu}_{AB}$.  Under Assumption \ref{ass: xi}, the effective dynamics \eqref{eq: effective dynamics} governing $Z_t$ is equivalent to the following dynamics proposed by \cite{maragliano2006string}: 
\begin{align}
        \widehat{Z}_t = -(M(\widehat{Z}_t)\nabla_{z}f(\widehat{Z}_t) + \beta^{-1}\nabla \cdot M(\widehat{Z}_t))\,dt + \sqrt{2\beta^{-1}}M(\widehat{Z}_t)^{1/2}\,dW_t. \label{eq: free energy dynamics}
    \end{align}
The equivalence between $Z_t$ and $\widehat{Z}_t$ has been noted in \cite{legoll2010effective} and it makes $Z_t$ easy to simulate by simulating $\widehat{Z}_t$ instead. This amounts to computing the diffusion tensor $M$ and the free energy, which can be done using well-known numerical methods \cite{maragliano2006string, Muller2021, Darve2008, Parrinello2017, barducci2008well}. 
By utilizing $M$ and $f$, the committor $\tilde{q}$ in collective variables can be computed as the solution to the following elliptic boundary value problem:
\begin{align}
\begin{cases}
    \beta^{-1}e^{\beta f(z)}\textsf{div}(e^{-\beta f(z)}M(z)\nabla \tilde{q}(z)) = 0, & z \in \tilde{\Omega}_{AB}, \\
    \tilde{q}(z) = 0, & z \in \tilde{A}, \\ 
     \tilde{q}(z) = 1, & z \in \tilde{B}.
        \end{cases}\label{eq: committor pde in CVs}
\end{align}
Here $\tilde{\Omega}_{AB} := \xi(\mathcal{M}) \setminus (\tilde{A} \cup \tilde{B})$. 
%Using $\tilde{q}$, $M$, and $f$ the low-dimensional 
The transition rate $\tilde{\nu}_{AB}$ in the CV space is given by: 
\begin{align}
    \tilde{\nu}_{AB} = \beta^{-1} Z^{-1}_{F} \int_{\tilde{\Omega}_{AB}} (\nabla \tilde{q})^{\top}M(z)\nabla \tilde{q} \exp{(-\beta f(z))}\,dz.
\label{eq: transition rate in CVs}
\end{align}
Computing $\nu_{AB}$ involves solving the high-dimensional committor problem \eqref{eq: committor pde} followed by high-dimensional numerical integration of \eqref{eq: transition rate}, which can be unfeasible or undesirable in practice. Instead, we can compute $\tilde{\nu}_{AB}$ by solving the low-dimensional BVP \eqref{eq: committor pde in CVs} and using the formula \eqref{eq: transition rate in CVs}. The discrepancy between the original transition rate $\nu_{AB}$ and the coarse-grained transition rate $\tilde{\nu}_{AB}$ is~\parencite{ZhangHartmannSchutte_2016}:
\begin{equation} 
\label{ZHS2016}
\nu_{AB} \le\tilde{\nu}_{AB}  = \nu_{AB} +\frac{1}{\beta}\int_{\Omega_{AB}}\nabla[q(x)-\tilde{q}(\xi(x))]^\top M(x)\nabla[q(x)-\tilde{q}(\xi(x))] \mu(x) dx.
\end{equation}
Thus, coarse-graining by a collective variable $\xi$ overestimates the transition rate unless $q(x) \equiv \tilde{q}(\xi(x))$.
%computed via effective dynamics \eqref{eq: effective dynamics}. 
It was also noted in \parencite{ZhangHartmannSchutte_2016} that the high-dimensional committor $q$ is the optimal CV. In fact, if $\xi = q$, then the effective dynamics reside in $[0,1]$. Moreover, for $z \in (0,1)$, the effective dynamics \eqref{eq: effective dynamics} become driftless with a diffusion coefficient $M(z) = \beta \nu_{AB}\exp{\left(\beta f(z)\right)}$. As a consequence, the committor is given by $\tilde{q}(z) = z$. Plugging these into \eqref{eq: transition rate in CVs}, we get $\tilde{\nu}_{AB} = \nu_{AB}$. Thus $q$ reproduces reaction rates \emph{exactly}. However, computing $q$ accurately by solving the high-dimensional PDE \eqref{eq: committor pde} is computationally prohibitive. Therefore, different methods are required to design CVs that minimally distort transition rates.

\section{Proposed methodology for learning CVs}
\label{sec: effective dynamics analysis}
The theory of quantitative coarse-graining reviewed above provides guidelines for CV design through analytical conditions for $\xi$ under which the error estimates \eqref{eq: relative entropy estimate} and \eqref{eq: pathwise distance} may hold and be minimized. Two such conditions are given by Assumption \eqref{ass: xi} which stipulates that $D\xi D\xi^{\top}$ is uniformly strictly positive definite and, \eqref{eq: orthogonality condition} and \eqref{eq: orthogonality condition 2} which make error estimates scale with $\epsilon$ in the potentials of the form $V = V_0 + \epsilon^{-1}V_1$, as in \eqref{eq: split potential}. In this section, we: 
\begin{enumerate}
    \item prove that \eqref{eq: orthogonality condition} and \eqref{eq: orthogonality condition 2} are equivalent, 
    \item numerically demonstrate that Assumption \ref{ass: xi} is not necessary for the reproduction of transition rates and may be relaxed to the requirement that $\|D\xi\|$ is bounded away from zero,
    %but may nonetheless require $D\xi$ to not entirely vanish, 
    and
    \item propose an algorithm based on \eqref{eq: orthogonality condition} for designing CVs for MD data. 
\end{enumerate}

\subsection{Conditions \eqref{eq: orthogonality condition} and \eqref{eq: orthogonality condition 2} are equivalent}

The theory on pathwise estimates and relative entropy provides two conditions, \eqref{eq: orthogonality condition} and \eqref{eq: orthogonality condition 2}, which may be used to improve the error estimates for effective dynamics and therefore to learn a collective variable. We now show that \eqref{eq: orthogonality condition} and \eqref{eq: orthogonality condition 2} are actually \emph{equivalent} and therefore the same condition provides simultaneous improvements in two different error metrics.

\begin{proposition}
\label{prop: oc poc equivalence}
    Let $\xi$ be a collective variable and $V$ be given by the decomposition $V= V_0 + \epsilon^{-1}V_{1}$ as in \eqref{eq: split potential}. Then \eqref{eq: orthogonality condition} and \eqref{eq: orthogonality condition 2} are equivalent. 
\end{proposition}

\begin{proof}
%    We show that \eqref{eq: orthogonality condition} $\iff$ \eqref{eq: orthogonality condition 2} for every $x \in \mathbb{R}^{N}$. %To that end, first to shorten notation we set
    Let $J := D\xi(x)$. 
    %Moreover 
    Recall that 
    \begin{align}
        \Phi = \left(JJ^{\top}\right)^{-1}. \notag
    \end{align}
    %Now, we compute 
    Then the entries of $I-\Pi$ are:
    \begin{align}
        [I - \Pi]_{lk} = \sum_{i,j} (\Phi^{-1})_{ij} (\nabla \xi_i)_{l} (\nabla \xi_j)_{k} = J^{\top}_{l}\Phi^{-1}J_{k}, %= \textsf{Tr}\left(J^{\top}_{l}\Phi^{-1}J_{k}\right) = \textsf{Tr}\left(\Phi^{-1}J_{k}J^{\top}_{l}\right), 
    \end{align}
    where $J_k$ and $ J_l$ are the $k$th and $l$th rows of $J$. Now, let $J = U_{r}\Sigma_{r}V^{\top}_{r}$ be the truncated SVD of $J$ where $r$ is the rank of $J$. Then $\Phi = \left(JJ^{\top}\right)^{-1} = U_{r}(\Sigma_{r}\Sigma_{r}^{\top})^{-1}U^{\top}_{r}$. Moreover, note that we can write each column of $J$ in terms of the columns of $U$ with coefficients from the rows of $V$. The $k$th column of $J$ is therefore:
    \begin{align}
        J_k = \sum_{\alpha=1}^{r}\sigma_{\alpha}v_{\alpha k}u_{\alpha}. 
    \end{align}
    Now we compute:
    \begin{align}
        \Phi^{-1}J_{k} &= U_{r}(\Sigma_{r}\Sigma_{r}^{\top})^{-1}U^{\top}_{r}\sum_{\alpha=1}^{r}\sigma_{\alpha}v_{\alpha k}u_{\alpha} \\
        &= \sum_{\alpha=1}^{r}\sigma_{\alpha}v_{\alpha k}U_{r}(\Sigma_{r}\Sigma_{r}^{\top})^{-1}U^{\top}_{r}u_{\alpha} \notag\\
        &= \sum_{\alpha=1}^{r}\sigma_{\alpha}v_{\alpha k}u_{\alpha}\sigma_{\alpha}^{-2} = \sum_{\alpha=1}^{r}\sigma_{\alpha}^{-1}v_{\alpha k}u_{\alpha}. \notag
    \end{align}
    Then $[I- \Pi]_{lk}$ is given by:  
    \begin{align}
 [I - \Pi]_{lk} &= J^{\top}_{l}\Phi^{-1}J_{k} = 
%        \textsf{Tr}
        \left(\sum_{\beta=1}^{r}\sigma_{\beta}v_{\beta l}u_{\beta}^{\top}\sum_{\alpha=1}^{r}\sigma_{\alpha}^{-1}v_{\alpha k}u_{\alpha}\right) \\
        &= \left(\sum_{1 \leq \alpha,\beta \leq r}\sigma_{\beta}v_{\beta l}\sigma_{\alpha}^{-1}v_{\alpha k}u_{\beta}^{\top}u_{\alpha}\right) = \left(\sum_{1 \leq \alpha \leq r}\sigma_{\alpha}v_{\alpha l}\sigma_{\alpha}^{-1}v_{\alpha k}\right) \notag\\
        &= \sum_{1 \leq \alpha \leq r}v_{\alpha l}v_{\alpha k} = \left[V_{r}V_{r}^{\top}\right]_{lk}. \notag
    \end{align}
    Thus, $[I - \Pi]_{lk} = \left[V_{r}V_{r}^{\top}\right]_{lk}$ so $I - \Pi = V_{r}V_{r}^{\top}$. 
    
    Now recall that \eqref{eq: orthogonality condition} states 
    \begin{equation*}
    D\xi(x)\nabla V_1(x) = J\nabla V_1 =  U_r\Sigma_rV_r^\top\nabla V_1 =0.
    \end{equation*}
    Since columns of $U_r\Sigma_r$ are linearly independent, the last equation is equivalent to
    \begin{equation*}
        V_r^\top\nabla V_1 = 0.
    \end{equation*}
    This, in turn, by the linear independence of columns of $V_r$, is equivalent to  $[I - \Pi] \nabla V_1 = V_rV_r^\top\nabla V_1 = 0.$
    Thus \eqref{eq: orthogonality condition} is equivalent to \eqref{eq: orthogonality condition 2}. 
    %By reversing the same argument we get that \eqref{eq: orthogonality condition 2} implies \eqref{eq: orthogonality condition}. 
\end{proof}

\subsection{The diffusion tensor may be rank-deficient but must not vanish}

\label{subsec: classical cvs for butane}

Assumption \eqref{ass: xi}, stating that $ D\xi D\xi^{\top} \succcurlyeq \delta^{2}_{\xi} I_{d}$ for a positive  constant $\delta^{2}_{\xi}$, ensures that the characterization of the conditional expectations through the co-area formula \eqref{eq: coarea formula} and therefore the measure $\mu^{\xi} \propto \exp(-\beta f(z))$ is an invariant measure for the effective dynamics~\eqref{eq: effective dynamics} which is equivalent to \eqref{eq: free energy dynamics} under Assumption \eqref{ass: xi}. However, in the absence of Assumption \eqref{ass: xi}, the diffusion tensor $M(z)$ may vanish, leading to the loss of ergodicity in the free energy dynamics \eqref{eq: free energy dynamics}. We describe such a case for butane below. 

\subsubsection{A chain-like butane molecule} The all-atom dynamics of the butane molecule $C_{4}H_{10}$ in a solvent under high friction can be modeled using \eqref{eq: OLD} with atomic coordinates $X_t$ residing in $\mathbb{R}^{14 \times 3}$ (see Figure \ref{fig: butane system}). Specifically, the process $\theta(X_t)$ largely resides in the anti ($\theta = \pi$) or gauche ($\theta = \pi/3, 5\pi/3$) states (Figure \ref{fig: butane free energy}) with rare transitions between them. The anti-gauche transition has therefore received considerable attention in molecular dynamics as a small-scale but significant example of torsional transitions and has served as a benchmark example for algorithms seeking to study conformational dynamics in larger biomolecules such as polymer chains. 
\begin{figure}[h]
    \centering
    \includegraphics[width=0.8\linewidth]{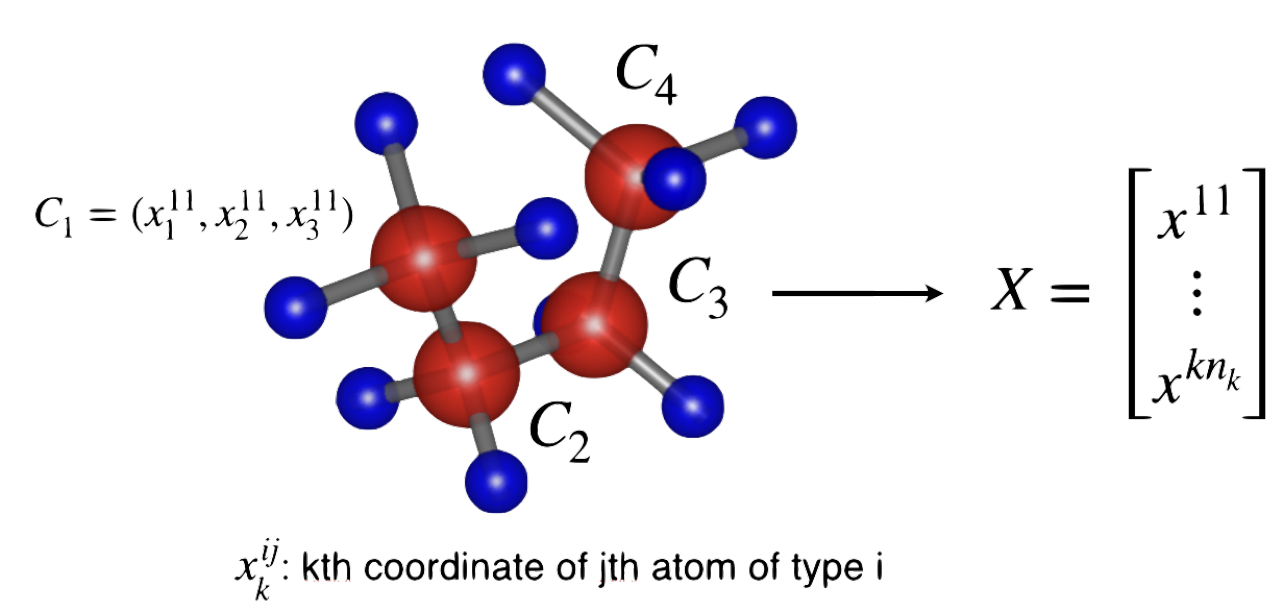}
    \caption{The butane molecule C${_4}$H$_{10}$ can be described in atomic coordinates in $\mathbb{R}^{42}$.}
    \label{fig: butane system}
\end{figure}
In practice, butane can be coarse-grained along the \emph{dihedral angle} $\theta$ in its carbon backbone, where it exhibits metastability. We simulate butane at 300K under Langevin dynamics with high friction, resulting in dynamics well-approximated by the overdamped regime \eqref{eq: OLD}. We provide additional details on this approximation in Appendix \ref{app: A} where we also describe how the high friction constant should be accommodated for in computing the transition rate via Equation \eqref{eq: transition rate in CVs}. 
\begin{figure}[h]
    \centering
    \includegraphics[width=\linewidth]{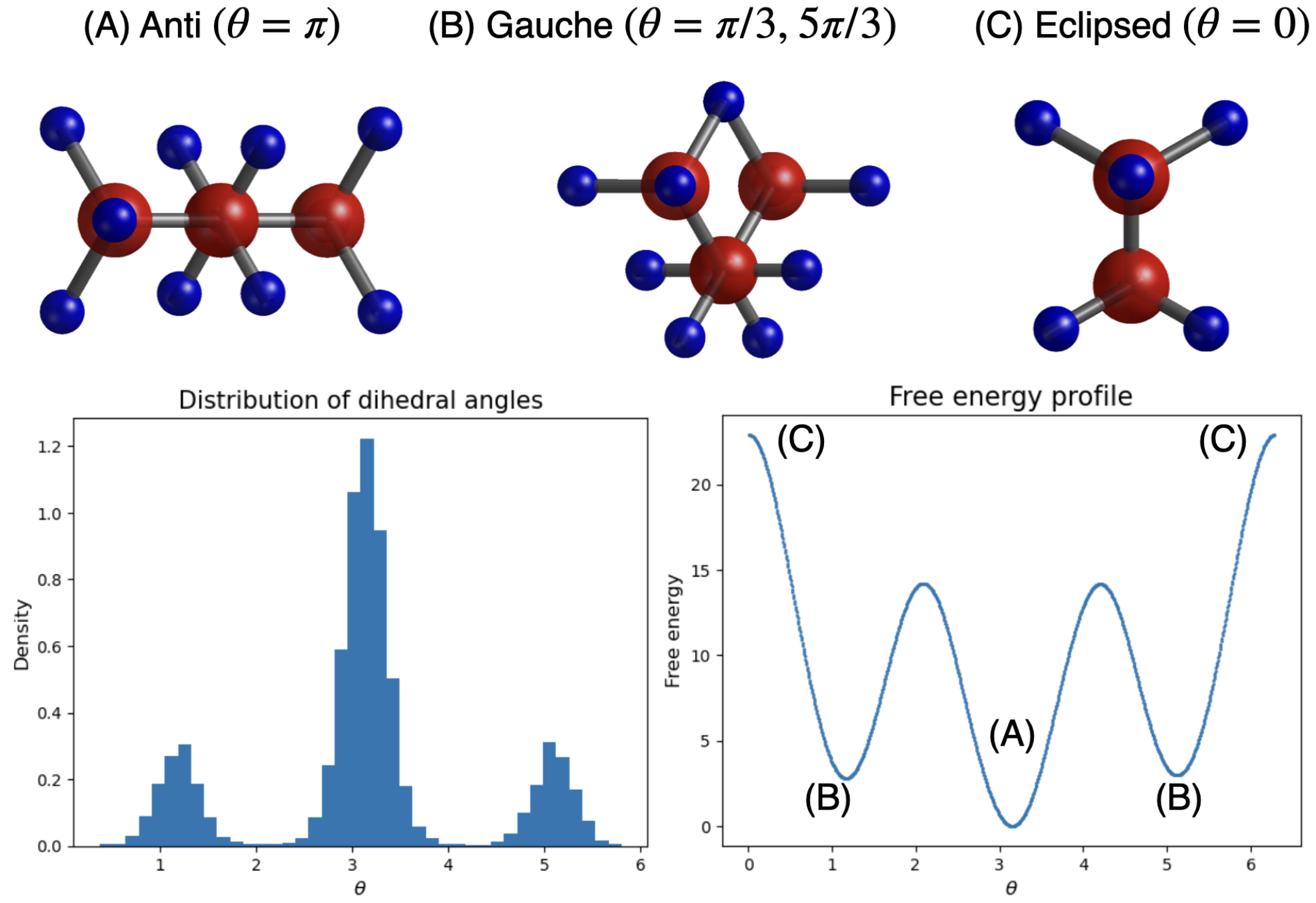}
    \caption{The dynamics of butane can be effectively coarse-grained by the dihedral angle $\theta$. In this case, the free energy $f$ has local minima near the anti and gauche states. Here, the system has been simulated at a temperature of 300 K.}
    \label{fig: butane free energy}
\end{figure}
The anti-gauche transition rate  obtained via brute-force all-atom simulation is $\nu_{AB} = 1.13 \pm 0.08 \times 10^{-2}$ ps$^{-1}$. Now we consider the following three collective variables based on the dihedral angle $\theta$ for reproducing the anti-gauche transition rate (Table \ref{tab: RatesTheta}):
\begin{enumerate}
    \item $\xi_1 = \theta(x)$: The collective variable $\xi_1$ reproduces the rate as $1.41 \times 10^{-2}$ ps$^{-1}$, within 3.5 standard deviations of the reference rate. Notably, the dihedral angle violates Assumption \eqref{ass: xi} since it is discontinuous at the point where $\theta = 0$ (i.e the location of the branch cut chosen for the angle). 
    \item $\xi_2 = \cos{\theta(x)}$: For the one-dimensional CV $\xi_2$ resulted in a transition rate of nearly 4.8 standard deviations of the reference rate, at $1.52 \times 10^{-2}$ ps$^{-1}$. We posit that this exageration occurs because both the free energy $F_2$ and one-dimensional diffusion tensor$A_2$ given by equations \eqref{eq: free energy} and \eqref{eq: effective_A} vanish (see central panel, Figure \ref{fig: classical cvs for butane}) leading to both the drift and diffusion terms in the free-energy dynamics \eqref{eq: free energy dynamics} to vanish. As a consequence, the dynamics are not irreducible and therefore non-ergodic. Moreover, the chosen gauche configurations form a disjoint set in the atomic dynamics but the CV $\xi_2$ merges these configurations into one state in the one-dimensional space. 
    \item $\xi_3 = (\sin{\theta(x)}, \cos{\theta}(x))$: We observe that for the two-dimensional CV $\xi_3$, the diffusion tensor $A_3$ has rank 1 (see right panel, Figure \ref{fig: classical cvs for butane}), but the transition rate is reproduced faithfully, at $1.19$ ps$^{-1}$, less than one standard deviation of the reference. This illustrates that Assumption \ref{ass: xi} is \emph{not necessary} for a CV to reproduce transition rates. Moreover, $\xi_3$ preserves the disjoint gauche configurations in CV space, unlike $\xi_2$. 
\end{enumerate}

\begin{figure}[h]
    \centering
    \includegraphics[width=\linewidth]{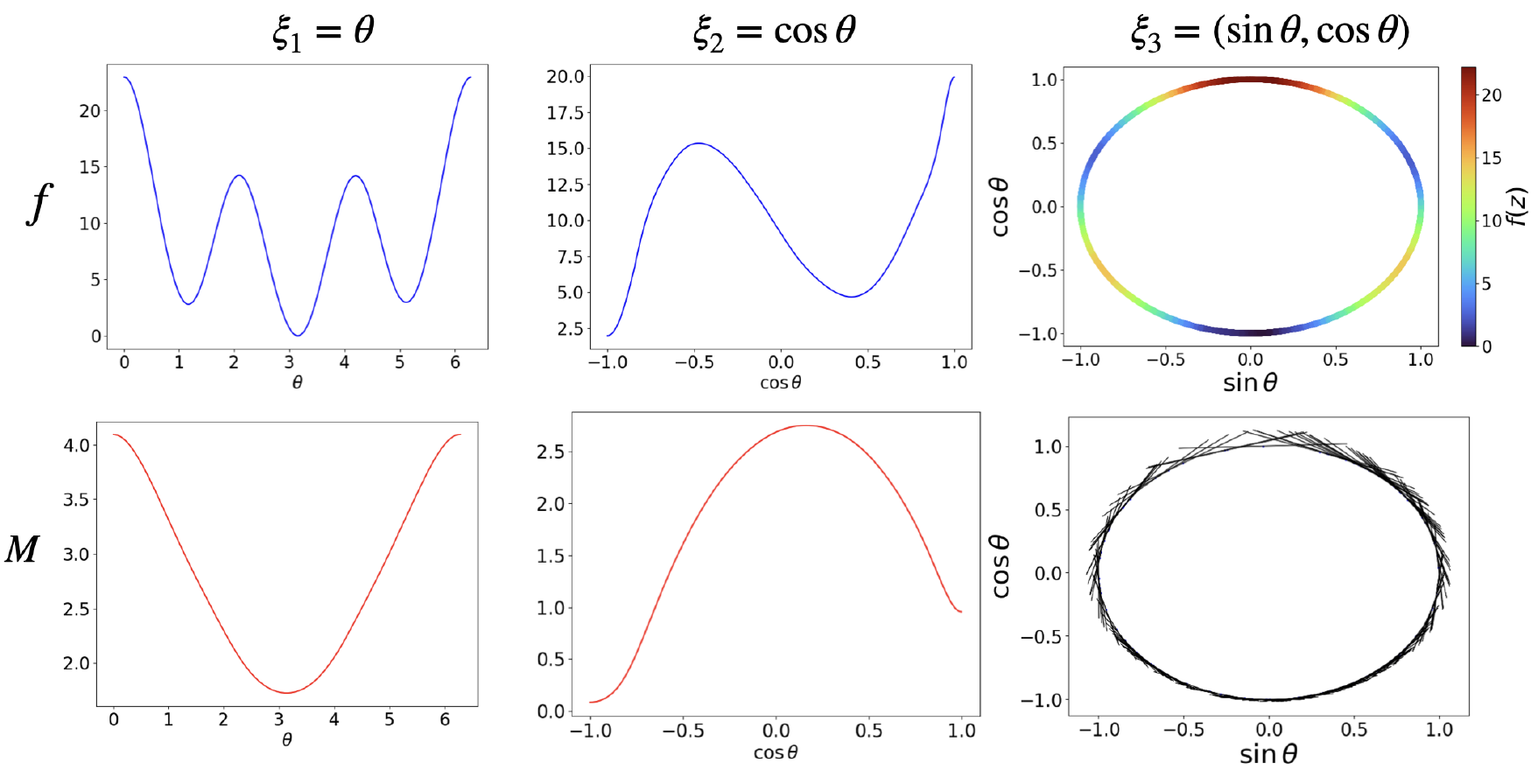}
    \caption{Free energy profiles (top) and diffusion tensors (bottom) for the collective variables based on the dihedral angle (left to right).}
    \label{fig: classical cvs for butane}
\end{figure}

\begin{table}[h!]
\begin{center}
\begin{tabular}{ll}
\toprule
Collective variable          & Transition rate (ps$^{-1}$)    \\
\midrule
Reference                    & 1.13 $\pm$ 0.08 $\times 10^{-2}$ \\
Dihedral angle $\theta$      & 1.41 $\times 10^{-2}$               \\
$(\sin \theta, \cos \theta)$ & $1.19 \times 10^{-2}$            \\
$\cos \theta$                & $1.52 \times 10^{-2}$             \\ 
\bottomrule
\end{tabular} \\

\caption{Collective variables for butane based on the dihedral angle, and the anti-gauche transition rates they yield.}
\label{tab: RatesTheta}

\end{center}
\end{table}

\subsection{Proposed algorithm} 
%We now propose learning a neural network CV $\xi$ through simulation data from overdamped Langevin dynamics \eqref{eq: OLD} 

 In this section, we propose an algorithm for machine-learning CV $\xi$ based on the orthogonality condition \eqref{eq: orthogonality condition}. Cautioned from the example in Section \eqref{subsec: classical cvs for butane}, we aim to ensure that $\xi$ not only preserves metastable states but also observes the condition that $D\xi \neq 0$ for any $x \in \mathcal{M}$. Among the many ways to ensure the non-existence of critical points, a proven effective strategy is to choose $\xi$ to be the encoder in an \emph{autoencoder} \cite{belkacemi2023autoencoders}. Additionally, to train such a CV satisfying OC, we would need knowledge of $\nabla V_1$ (the gradient of the confining potential), i.e, the normal vector to the residence manifold $\mathcal{M}$. However, the key obstruction to this approach is that a priori neither the confining term $V_1$ in \eqref{eq: split potential} nor the residence manifold $\mathcal{M}$ are known. In this work, we will resolve this obstruction by first obtaining a parameterization $\Psi: \mathbb{R}^{N} \to \mathbb{R}^{D+1}$ which projects $\mathcal{M}$ into a $D$ dimensional hypersurface $\widehat{\mathcal{M}} := \Psi(\mathcal{M})$. $\widehat{\mathcal{M}}$ will therefore serve as a surrogate for $\mathcal{M}$ and its property as a hypersurface enables learning a surrogate potential $\widehat{\Phi}$ vanishing on $\widehat{\mathcal{M}}$ with gradients aligned normally to the $\widehat{\mathcal{M}}$. Finally, we learn $\widehat{\xi}$ which satisfies the \eqref{eq: orthogonality condition} $\nabla \widehat{\xi} \cdot \nabla \widehat{\Phi} = 0$. Notably, this $\widehat{\xi}$ can be \emph{lifted} to an overall collective variable via composition with the parameterization, given by $\xi := \widehat{\xi} \circ \Psi$. Therefore, our method can be described as recovering a CV by enforcing the orthogonality condition in a latent space $\mathbb{R}^{D+1}$ instead of the ambient space $\mathbb{R}^{N}$. We summarize our method in Algorithm \eqref{alg: learning cv} and elaborate on each step of this method in further sections below. 

\begin{algorithm*}
\label{alg: learning cv}
     \caption{Learning CV that satisfies the orthogonality condition \eqref{eq: orthogonality condition}}
     \KwInput{Data $\mathcal{X} = \{x_i\}_{i=1}^{n} \subseteq \mathcal{M}$, invariant density $\mu_i \propto \exp{\left(-\beta V(x_i) \right)}$.}
     
     \textbf{Step 1: } Learn the residence manifold $\mathcal{M}$ as $\widehat{\mathcal{M}} \subseteq \mathbb{R}^{D+1}$ via a \emph{diffusion map} \cite{coifman2006diffusion} accompanied with a \emph{diffusion net} \cite{mishne2019diffusion} learned through Algorithm ~\ref{alg: ResManLearn} $$\Psi: \mathcal{M} \to \widehat{\mathcal{M}} \subseteq \mathbb{R}^{D+1}$$. 
     
     \textbf{Step 2: } Learn the surrogate potential: Train a neural network $$\widehat{\Phi}: \mathbb{R}^{D+1} \to \mathbb{R}$$ such that $\widehat{\Phi}$ vanishes on $\widehat{\mathcal{M}}$. 
     
     \textbf{Step 3: } Train an autoencoder where the encoder $\widehat{\xi}: \mathbb{R}^{D+1} \to \mathbb{R}$ satisfies $$\nabla \widehat{\xi} \cdot \nabla \widehat{\Phi} = 0$$. 
     
     \KwOutput{Collective variable $\xi = \widehat{\xi} \circ \Psi$.} 
\end{algorithm*}

We remark that Step 1, which involves computing the low-dimensional embedding, is the most technically intricate subroutine of this algorithm. We summarize the subroutine in Algorithm ~\ref{alg: ResManLearn} and provide a detailed walkthrough in the section below.

\section{Learning the residence manifold} 
\label{sec: resmanlearn}

Our goal is to learn a function $\Psi: \mathbb{R}^{N} \to \mathbb{R}^{D+1}$ such that $\mathcal{M}$ may be represented by the surrogate manifold $\widehat{\mathcal{M}} = \Psi(\mathcal{M})$. We will learn this map from a trajectory $\mathcal{X} = \{x_i\}_{i=1}^{n}$ where $x_i \sim X_{i\Delta T}$, i.e the $i$th point of a time-series simulation of \eqref{eq: OLD}. We will adopt diffusion maps \parencite{coifman2008diffusion} to learn this map $\Psi$, but this cannot be done straightforwardly on data sampled from the process $X_t$ due to a number of obstacles arising from the nature of MD simulations and the requirements of our proposed algorithm \eqref{alg: learning cv}. We highlight the key  obstacles below:
\begin{enumerate}
    \item \emph{Group invariance:} In typical MD simulations, the confining potential $V_{1}$ is a combination of several strong pairwise energies such as covalent interactions between atoms and is therefore invariant to actions of SE(3), the group of global translations and rotations of $\mathbb{R}^3$ (see Definition \eqref{def: groups}). Therefore, if $X \in \mathcal{M}$, then $gX \in \mathcal{M}$ where $g$ is a representation of an element in SE(3). Taking the action over all group elements, the orbit of any point $x$ under SE(3) actions forms a submanifold of $\mathcal{M}$. But this is not the underlying slow direction we would like to discover. 
    \item \emph{Hypersurface learning:} Since our goal is to learn a normal vector field on $\Psi(\mathcal{M})$, we will learn the residence manifold as a hypersurface (see Definition \eqref{def: hypersurface}) in $\mathbb{R}^{D+1}$. Not every manifold is embeddable as a hypersurface, so we only hope to achieve this hypersurface locally, allowing for possible self-intersections, knots, or lack of orientability. Therefore, we will only attempt to embed the residence manifold as an \emph{immersed} hypersurface.  
    \item \emph{Out of sample extension:} Typically, manifold learning algorithms will recover a map $\psi: \mathcal{X} \to \mathbb{R}^{D+1}$. We require that our proposed map $\Psi$ is an extension of $\psi$ to $\mathcal{M}$. 
\end{enumerate} 
We surmount the above obstacles as follows:
\begin{enumerate}
    \item \emph{Group invariance:} To account for SE(3) invariances, we use \emph{group invariant features} $\mathcal{F}: \mathbb{R}^{N} \to \mathbb{R}^{d_{\mathcal{F}}}$ satisfying $\mathcal{F}(x) = \mathcal{F}(gx)$ for all $g \in G$. To use group invariant features, we first project the raw data $\mathcal{X}$ onto a \emph{feature space} by computing a \emph{featurized point cloud} $\mathcal{X}_{\mathcal{F}} := \{\mathcal{F}(x_i)\}$. Assuming that $\mathcal{X}_{\mathcal{F}}$ is sampled from a \emph{featurized manifold} $\mathcal{M}_{\mathcal{F}}$, we compute a diffusion map $\psi_{\mathcal{F}}$ on $\mathcal{X}_{\mathcal{F}}$ instead of $\mathcal{X}$. We describe our choices of feature map in Section \eqref{subec: group invariant features}. 
    \item \emph{Hypersurface learning:} We apply the diffusion map algorithm to $\mathcal{X}_{\mathcal{F}}$ and obtain $\psi_{\mathcal{F}}: \mathcal{X}_{\mathcal{F}} \to \mathbb{R}^{D+1}$ such that $\psi_{\mathcal{F}}(\mathcal{X}_{\mathcal{F}})$ appears sampled from a hypersurface immersed in $\mathbb{R}^{D+1}$. To achieve such an immersion (see Definition \eqref{def: immersion}) we use independent eigencoordinate selection \parencite{chen2019selecting}.  
    \item \emph{Out of sample extension:} Once $\psi_{\mathcal{F}}$ has been obtained, we extend it to $\mathcal{M}_{\mathcal{F}}$ via \emph{diffusion nets} \parencite{mishne2019diffusion}, obtaining a neural network $\Psi^{\mathcal{F}}$. The overall parameterization will then be given by $\Psi := \Psi^{\mathcal{F}}\circ \mathcal{F}$. 
\end{enumerate}

\subsection{Data post-processing matters} \label{subec: group invariant features}
We contrast eight different choices of feature maps $\mathcal{F}$. We will assume $x \in \mathbb{R}^{N}$ is a point in all-atom space where $x$ is a single vector of all 3D positions of all atoms. This all atom vector may be reshaped into the all-atom matrix $X \in \mathbb{R}^{(N/3) \times 3}$ where the $n$th row in $X$ is given by $[x_{3n}, x_{3n+1}, x_{3n+2}]$. With this notational change, we describe the feature maps below: 
\begin{enumerate}
\item No featurization ({\sf NoFeaturization}): $\mathcal{F}(X) = X$. 
\item {\sf Recentering}: Here $\mathcal{F}(X) = X - \mu(X)$, where $\mu(X)$ is the mean of the rows of the atomic coordinates $X \in \mathbb{R}^{14 \times 3}$.
    \item Trajectory alignment ({\sf TrajAlign}): Aligning each configuration $X_i$ at time-step $t_i$ to the previous configuration $X_{i-1}$ via Procrustes alignment). {\sf TrajAlign} is a common method for post-processing MD trajectories but it does not correspond to an unambiguous function independent of the data. Consequently, this map is difficult to compute at inference time where a new configuration not in the MD trajectory must be featurized. 
    \item Computing the gram-matrix of the recentered configurations ({\sf GramMatrix}): $\mathcal{F}(X) = (X - \mu(X))(X-\mu(X))^{\top} \in \mathbb{R}^{(N/3) \times (N/3)}$,
    \item Computing the gram-matrix of the recentered configurations of the carbon atoms ({\sf GramMatrixCarbons}): $\mathcal{F}(X) = (MX - \mu(MX))(MX-\mu(MX))^{\top} \in \mathbb{R}^{4 \times 4}$. Here $M \in \mathbb{R}^{4 \times 14}$ is a \emph{mask} which selects the indices corresponding to the carbon atoms. 
    \item Global translation of the atoms to realign $C_1$ to the origin followed by rotation of the $C_1-C_2$ bond onto the $x$-axis ({\sf BondAlign(1,2)}), 
    \item Global translation of the atoms to realign $C_2$ to the origin followed by rotation of the $C_2-C_3$ bond onto the $x$-axis ({\sf BondAlign(2,3)}),
    \item Changing basis of the recentered configurations such that the $C_2, C_3, C_4$ atoms all lie in the $xy$-plane ({\sf  PlaneAlign}). 
\end{enumerate}
The above feature maps represent changing levels of group invariances, ranging from no invariance to complete invariance to rotations, reflections, and translations (E(3) invariance). In particular, {\sf TrajAlign} is not a feature map but nonetheless a commonly used pre-processing technique. {\sf NoFeaturization} has no invariance, {\sf Recentering} only has translation invariance, and {\sf GramMatrixCarbons} is only invariant to the actions of E(3) on a subspace of $\mathbb{R}^{N}$ spanned by the carbon atoms. {\sf BondAlign(1,2)}, {\sf BondAlign(2,3)}, and {\sf PlaneAlign} are SE(3) invariant but not E(3) invariant since they separate antipodal configurations. Finally, {\sf GramMatrix} is E(3) invariant as is well known from Hilbert's classical invariant theory. We visualize {\sf BondAlign(1,2)}, {\sf BondAlign(2,3)}, and {\sf PlaneAlign} in Figures \ref{fig: bondalign 1,2}, \ref{fig: bondalign 2,3}, and \ref{fig: PlaneAlign} respectively.

\begin{figure}[h]
    \centering
    \includegraphics[width=0.8\linewidth]{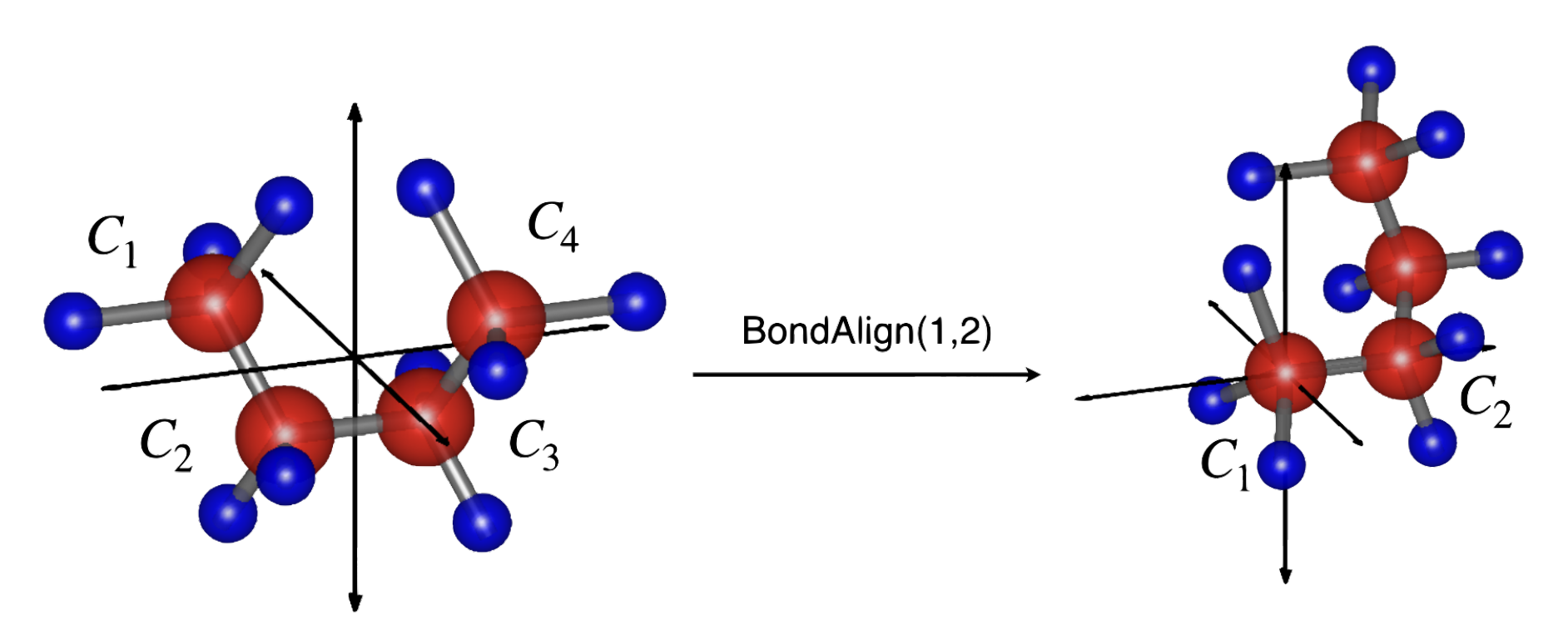}
    \caption{In {\sf BondAlign(1,2)}, the $C_1-C_2$ bond is aligned along the $x$-axis.}
    \label{fig: bondalign 1,2}
\end{figure}

\begin{figure}[h]
    \centering
    \includegraphics[width=0.8\linewidth]{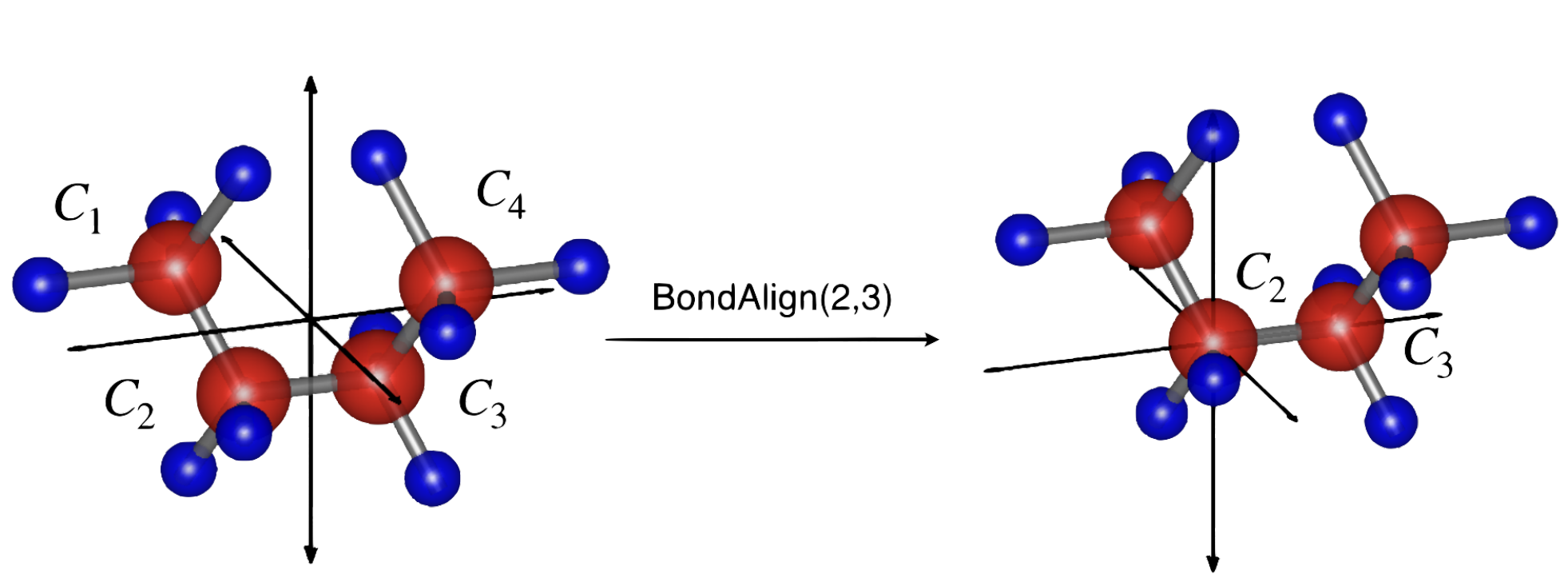}
    \caption{In {\sf BondAlign(2,3)}, the $C_2-C_3$ bond is aligned along the $x$-axis.}
    \label{fig: bondalign 2,3}
\end{figure}

\begin{figure}[h]
    \centering
    \includegraphics[width=0.8\linewidth]{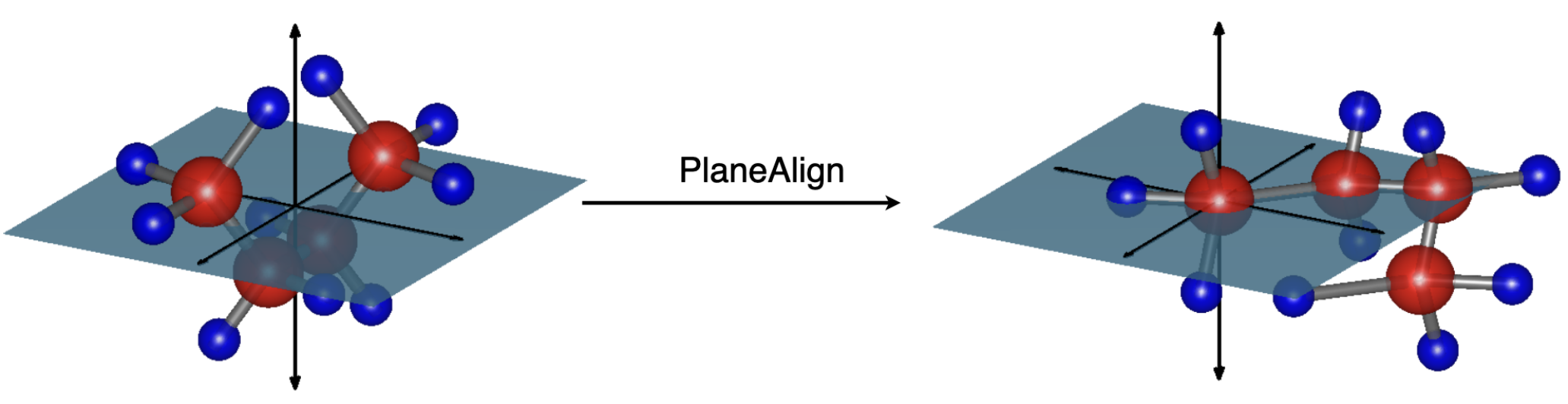}
    \caption{In {\sf PlaneAlign} the basis is changed such that the first three carbon atoms are in the $xy$-plane.}
    \label{fig: PlaneAlign}
\end{figure}

\subsection{Diffusion maps} 

We pass the data $\mathcal{X}$ through $\mathcal{F}$ to get a featurized pointcloud $\mathcal{X}_{\mathcal{F}} = \{p_i\}$ where $p_i = \mathcal{F}(x_i)$, which we posit is sampled from the featurized manifold $\mathcal{M}_{\mathcal{F}} := \mathcal{F}(\mathcal{M})$. To embed the featurized manifold in low dimensions we use diffusion maps \parencite{coifman2008diffusion}. Our choice is motivated by the fact that if $\Delta_{\mathcal{M}_{\mathcal{F}}}$ is the Laplace-Beltrami operator on $\mathcal{M}_{\mathcal{F}}$ and $\tilde{\psi}_i$ are the Laplacian eigenfunctions satisfying: 
\begin{align}
    -\Delta_{\mathcal{M_{\mathcal{F}}}}\tilde{\psi}_i = \lambda_i \tilde{\psi}_i,
\end{align}
then there exists an $m \in \mathbb{N}$ such that the map
\begin{align}
    \tilde{\psi}(x) :=  (\lambda_{1}\tilde{\psi}_1(x), \ldots, \lambda_{m}\tilde{\psi}_{m}(x)) \in \mathbb{R}^{m} \label{eq: Laplacian embedding}
\end{align}
is an \emph{embedding} \cite{BATES2014516}. This $m$ can be chosen to be sufficiently larger than the intrinsic dimension $\textsf{dim}(\mathcal{M}_\mathcal{F})$ of the featurized manifold. The diffusion map $\tilde{\psi}^{n,\epsilon}$ is in turn an approximation of the Laplacian embedding \eqref{eq: Laplacian embedding} when the density rescaling parameter is set to 1 in the original diffusion map algorithm \cite{coifman2008diffusion}. In this setting, $\tilde{\psi}^{n,\epsilon}$ is computed as follows: first, using a Gaussian kernel $k_{\epsilon}(p,q) = \exp(-\|p-q\|^{2}_{2}/\epsilon)$ tuned by a bandwidth $\epsilon$, we form the kernel matrix $K^{\epsilon}_{ij} = k_{\epsilon}(p_i,p_j)$ and calculate the kernel density estimate as the vector of its row means: 
 \begin{align}
     \rho^{n,\epsilon}(i) = \frac{1}{n}\sum_{j=1}^n K^{\epsilon}_{ij}.
 \end{align}
Next, we right-normalize the kernel matrix as
\begin{align}
K^{n,\epsilon} = K_{\epsilon}[D^{n,\epsilon}]^{-1},\quad D^{n,\epsilon} = {\sf diag}\{\rho^{n,\epsilon}(1),\ldots,\rho^{n,\epsilon}(n)\}.    
\end{align}
The right division by $D$ removes the effect of the sampling density $\rho$ and therefore normalizes the data such that it behaves as though it were sampled via the uniform distribution on $\mathcal{M}_{\mathcal{F}}$. We convert $K_{\epsilon,1}$ into a stochastic matrix of a Markov chain by left-normalizing $K_{\epsilon}$ to make its row sums equal to 1:
$$
P^{n,\epsilon} = [T^{n,\epsilon}]^{-1}K^{n,\epsilon}, \quad T^{n,\epsilon} = {\sf diag}(K^{n,\epsilon}\mathbf{1}_n). 
$$
Thus, using the kernel $k_{\epsilon}$, we have constructed a Markov chain on the data $\mathcal{X}$ whose transition matrix is given by $P_{\epsilon}$. Then, we form the generator: 
\begin{align}
    L^{n,\epsilon} = \frac{1}{\epsilon}\left(I - P^{n,\epsilon}\right). \label{eq: tmdmap generator}
\end{align}
Notably, for $f \in C^2(\mathcal{M}_{\mathcal{F}})$ and fixed $p \in \mathcal{M}_{\mathcal{F}}$, let $\mathbf{f} = [f(p_i)]_{p_i \in \mathcal{X}_{\mathcal{F}} \cup \{p\}}$. Then $L^{n,\epsilon}\mathbf{f}(p) \to \Delta_{\mathcal{M}_{\mathcal{F}}}f(p)$ as $n \to \infty$ and $\epsilon \to 0$, so $L^{n, \epsilon}$ is a meshless discretization of $\Delta_{\mathcal{M}_{\mathcal{F}}}$. Moreover, eigenvectors $\tilde{\psi}^{n,\epsilon}_j$ of $L^{n,\epsilon}$ are discretizations of $\tilde{\psi}_j$. Thus, the diffusion map $\tilde{\psi}^{n,\epsilon}: \mathcal{X}_{\mathcal{F}} \to \mathbb{R}^m$ is given by 

\begin{align}
    \tilde{\psi}^{n,\epsilon}(p_i) := \left(\lambda^{n,\epsilon}_{1}\tilde{\psi}^{n,\epsilon}_{1}(p_i), \ldots, \lambda^{n,\epsilon}_{m}\tilde{\psi}^{n,\epsilon}_{m}(p_i)\right).
\end{align}

In an abuse of notation we identify the map $\tilde{\psi}^{n,\epsilon}$ with the $n \times m$ matrix formed by stacking the rows $\tilde{\psi}^{n,\epsilon}(p_i)$. We summarize diffusion maps in Algorithm \eqref{alg: DMap}. 

\begin{algorithm}[h]
    \caption{Diffusion Maps}
    \KwIn{Dataset $\{p_i\}_{i=1}^{N} \subset \mathbb{R}^{d_{\mathcal{F}}}$, bandwidth $\epsilon > 0$, kernel function $k_{\epsilon}(x, y)$, number of eigenvectors $m$.}
    
    \textbf{Compute Affinity Matrix}
    $$K^{\epsilon}_{ij} = k_{\epsilon}(p_i,p_j)$$
    
    \textbf{Construct KDE}
    $$\rho^{n,\epsilon}(i) = \frac{1}{n}\sum_{j=1}^n K^{\epsilon}_{ij}.$$
    
    \textbf{Renormalize kernel}
    $$K^{n,\epsilon} = K_{\epsilon}[D^{n,\epsilon}]^{-1},\quad D^{n,\epsilon} = {\sf diag}\{\rho^{n,\epsilon}(1),\ldots,\rho^{n,\epsilon}(n)\}.$$
        
    \textbf{Construct Laplacian}
    $$L^{n,\epsilon} = \frac{1}{\epsilon}\left(I - P^{n,\epsilon}\right), \quad P^{n,\epsilon} = [T^{n,\epsilon}]^{-1}K^{n,\epsilon}, \quad T^{n,\epsilon} = {\sf diag}(K^{n,\epsilon}\mathbf{1}_n). $$

    \textbf{Construct diffusion map}
    $$ L^{n,\epsilon} \psi^{n,\epsilon} = \Lambda^{n,\epsilon}\psi^{n,\epsilon}$$
    
    \KwOut{Diffusion map embedding $\Lambda^{n,\epsilon}\psi^{n,\epsilon}$}
    \label{alg: DMap}
\end{algorithm}

\subsection{Independent eigencoordinate selection (IES)}

While $\tilde{\psi}^{n,\epsilon}$ approximates the immersion $\tilde{\psi}$ of $\mathcal{M}_{\mathcal{F}}$ in $\mathbb{R}^{m}$, this $m$ might be too large for the immersed manifold to be a hypersurface. We will therefore use IES \parencite{chen2019selecting} to select a subset of columns $S$ of the matrix $\tilde {\psi}^{n,\epsilon}$, $|S| = D+1$, such that the resulting map 
\begin{align}
    \psi_{\mathcal{F}} = (\tilde{\psi}^{n,\epsilon}_j(p))_{j \in S}
\end{align}
has rank $D$. To introduce IES, we need background on differential geometry found in Appendix \ref{app: B}. We assume that $\mathcal{M}_{\mathcal{F}}$ is a Riemannian manifold of dimension $\text{dim}(\mathcal{M}_{\mathcal{F}})$ and metric $\{g(p)\}_{p \in \mathcal{M}_{\mathcal{F}}}$. Let $\tilde{\psi}: \mathcal{M}_{\mathcal{F}} \to \mathbb{R}^m$ be a smooth immersion and $d\tilde{\psi}_{p}: T_{p}\mathcal{M}_{\mathcal{F}} \to T_{\tilde{\psi}(p)}\mathbb{R}^{m}$ be its differential at $p \in \mathcal{M}_{\mathcal{F}}$. Note that since $\tilde{\psi}$ is an immersion, $\text{rank}(d\tilde{\psi}) = \text{dim}(\mathcal{M}_{\mathcal{F}})$. However, $d\tilde{\psi}_{p}$ may not be full rank since $m \geq \text{dim}(\mathcal{M}_{\mathcal{F}})$. We can nonetheless compute its pseudo-inverse on vectors in $T_{\tilde{\psi}(p)\mathbb{R}^{m}}$ to define the \emph{pushforward metric} $g_{*}$: 
\begin{align}
    \langle u,v \rangle_{g_{*}} := \langle d\tilde{\psi}_{p}^{\textdagger} u, d\tilde{\psi}_{p}^{\textdagger} v\rangle_{g(p)}, \: u,v \in T_{\tilde{\psi}(p)}\mathbb{R}^{m}. \label{eq: pseudometric}
\end{align}
The pushforward metric gives a Riemannian structure to $\tilde{\psi}(\mathcal{M})$ that makes $\tilde{\psi}$ an isometry by definition. In local coordinates in $\mathbb{R}^{m}$, this results in a quadratic form at each $\tilde{\psi}(p_i)$: 
\begin{align}
    \langle u,v \rangle_{g_{*}} = u^{\top}G(p_i)v. 
\end{align}
Here $G(p_i) \in \mathbb{R}^{m \times m}$ is a local representation of the metric. Crucially, since $\text{rank}(d\tilde{\psi}_p) < m$, $G(p_i)$ will be positive semi-definite with $\text{rank}(G(p_i)) = \text{dim}(\mathcal{M}_{\mathcal{F}}) = D$. Thus, the first $D$ singular vectors of $G$ will span $T_{\tilde{\psi}(p)}\tilde{\psi}(\mathcal{M}_{\mathcal{F}})$, while the $(D+1)$th singular vector will be contained in $T_{\tilde{\psi}(p)}\tilde{\psi}(\mathcal{M}_{\mathcal{F}})^{\perp}$. Remarkably, these $G(p_i)$ can be estimated from the data $\mathcal{X}_{\mathcal{F}}$ via Algorithm~\ref{alg: rmetric-alg} introduced in~ \cite{perraul2013non} and named RMetric. 

\begin{algorithm}
    \SetKwInOut{Input}{Input}
	\SetKwInOut{Output}{Return}
	\SetKwComment{Comment}{$\triangleright $\ }{}
    \Input{Embedding $\tilde{\psi}^{n,\epsilon} \in \mathbb R^{n\times m}$, Laplacian $L^{n,\epsilon}$, target dimension $D+1$}
    \For{all $i = 1\to n, \alpha = 1\to m, \beta = 1\to m$}{
		$[\tilde{H}(i)]_{\alpha\beta} = \sum_{j\neq i} L^{n,\epsilon}_{ij} (\tilde{\psi}^{n,\epsilon}_{\beta}(p_j) - \tilde{\psi}^{n,\epsilon}_{\beta}(p_i))(\tilde{\psi}^{n,\epsilon}_{\alpha}(p_j) - \tilde{\psi}^{n,\epsilon}_{\alpha}(p_i))$
    }
    \For{$i = 1\to n$}{
		$U(i)$, $\Sigma(i) \gets $ {\sc ReducedRankSVD}$(\tilde{H}(i), D+1)$ \\
		$H(i) = U(i)\Sigma(i)U(i)^\top$\\
    $G(i) = U(i)\Sigma^{-1}(i)U(i)^\top$ \\
    }
    \Output{$G(i),H(i) \in \mathbb R^{ m\times m}$, $U(i)\in \mathbb R^{m\times (D+1)}$,  $\Sigma(i)\in \mathbb R^{(D+1)\times (D+1)}$, $i \in\{1,2,\ldots,n\}$}
    \caption{RMetric}
    \label{alg: rmetric-alg}
\end{algorithm}

Using RMetric, we compute these singular vectors $U(p_i) \in \mathbb{R}^{m \times (D+1)}$, noting that $m$ is the number of eigenfunctions computed and $D+1$ is the suggested dimension in which to eventually embed the hypersurface. Intuitively, a $D$-dimensional hypersurface (locally) has large $ D$-dimensional volume but has small $ (D+1)$-dimensional volume. Following this intuition, selecting $S \subseteq [m]:=\{1,2,\ldots,m\}$ thus amounts to selecting rows of $U(p_i)$ with large $D$-dimensional volume but small $(D+1)$-dimensional volume. We quantify this in the following way: suppressing the dependence on $p_i$ for a moment, let $U' = U(p_i)[:,[D]]$ be the $D$ column vectors which span $T_{\tilde{\psi}(p_i)}\tilde{\psi}(\mathcal{M}_{\mathcal{F}})$. Next let $U'_{S'} = U[S',:]$ be a selection of the rows of $U'$ where $|S'| = D$. We consider the log of the volume of the parallelepiped formed by the $D$-dimensional rows of the column-normalized $U'_{S'}$, written as follows: 
\begin{align}
    \text{Vol}(U'_{S'},p_i) = \log \left( \sqrt{\det \left(\left[U'_{S'}\right]^{\top}U'_{S'}\right)}\right) -\sum_{j=1}^{D}\log\|U'_{S'}[:,j]\|_{2}^{2}. \label{eq: S dim volume}
\end{align}
To estimate this volume over all points $p_i$, we take the mean of the normalized projected volume over all the data. Moreover, since $S'$ corresponds to a selection of Laplacian eigenvectors, we also promote selecting eigenvectors with greater smoothness. To do so, we add a regularization term $-\sum_{k \in S'}\lambda_k$ weighted by a parameter $\zeta$ which penalizes how much the selected eigenvectors oscillate. Note that we take a minus sign in front since we aim to \emph{maximize} the volume score with the \emph{lowest} frequency eigenfunctions. This leads to the overall score considered in IES \parencite{chen2019selecting}: 
\begin{align}
    R_{\zeta}(S') = \frac{1}{n}\sum_{i=1}^{n}\text{Vol}(U'_{S'},p_i) - \zeta\sum_{k \in S'}\lambda_k. \label{eq: overall volume score}
\end{align}

In the original implementation of IES, $R_{\zeta}(S')$ is maximized over all $S' \subseteq [m]$ with $ |S| = D$ to obtain an optimal set $S_*$. Here, we simply take $S_* = [D]$, the first $D$ eigenfunctions.  This algorithm is summarized in Algorithm \ref{alg: IES}. 

\begin{algorithm}
    \SetKwInOut{Input}{Input}
    \SetKwInOut{Output}{Return}
    \SetKwComment{Comment}{$\triangleright $\ }{}
 
    \Input{Tangent bundle $\{U'(p_i)\}_{p_i \in \mathcal{X}_{\mathcal{F}}} \subseteq \mathbb{R}^{m \times (D)}$, target dimension $|s| \leq D$, regularization parameter $\zeta$.}
    \For{$S' \subseteq [m]: |S'| = s, 1\in S'$}{
        $R_{\zeta}(S') \gets 0$ \\
        \For{$i = 1,\cdots, n$}{
            $U'_{S'} \gets U(p_i)[S', :]$ \\
            $R_{\zeta}(S') \gets R_{\zeta}(S') + n^{-1}\left[\log \left( \sqrt{\det \left(\left[U'_{S'}\right]^{\top}U'_{S'}\right)}\right) - \sum_{j=1}^{s}\log\|U'_{S'}[:,j]\|_{2}^{2} - \zeta\sum_{k \in S'}\right]. $
        }
        
    }
    $S^* = {\text{argmax}} R_{\zeta}(S')$ {\tt \# Alternative: S* = [D]} \\
    \Output{Independent eigencoordinates set $S^*$}
    \caption{Independent eigencoordinate search (IES)}
    \label{alg: IES}
\end{algorithm}

% Once $S_*$ has been selected, we let $S = S_* \cup \{k\}$ {\color{red} SHASHANK: what is $k$?} and consider $U_{S} := U[S, :]$ and attempt to \emph{minimize} $\text{R}(U_S,-\zeta)$. The justification behind this is as follows: when selecting $S_*$ we found the eigenvectors with the largest $D$-dimensional volume. Now to these selected eigenvectors we add an eigenvector such that the resulting coordinates have minimal $D+1$-dimensional volume, thus fulfilling the criteria for being a hypersurface. Moreover, we multiply $\zeta$ by $-1$ since we want to find the lowest frequency eigenfunction, which makes the data a hypersurface. After selecting such an optimal $k_*$, we take $S = S_* \cup \{k_*\}$ as the final selection of $D+1$ eigenvectors. We term this addition to IES as Hypersurface Search (HyperSearch) and summarize it in Algorithm \eqref{alg: HyperSearch}.

\begin{algorithm}
    \SetKwInOut{Input}{Input}
    \SetKwInOut{Output}{Return}
    \SetKwComment{Comment}{$\triangleright $\ }{}
 
    \Input{Tangent bundle $\{U(p_i)\}_{p_i \in \mathcal{X}_{\mathcal{F}}} \subseteq \mathbb{R}^{m \times (D+1)}$, Optimal $D$ independent coordinates $S^*$, regularization parameter $\zeta$.}
    \For{$k \in [D+1] \setminus S^*$}{
        $S \gets S^* \cup \{k\}, \quad R(S, \zeta) \gets 0$ \\
        \For{$i = 1,\cdots, n$}{
            $U_{S} \gets U(p_i)[S, :]$ \\
            $R_{-\zeta}(S) \gets R_{-\zeta}(S) + n^{-1}\left[\log \left( \sqrt{\det \left(\left[U_{S}\right]^{\top}U_{S}\right)}\right) - \sum_{j=1}^{s}\log\|U_{S}[:,j]\|_{2}^{2} + \zeta\sum_{k \in S}\lambda_k\right]. $
        }
        
    }
    $S \gets {\text{argmin}} \: R_{-\zeta}(S)$ {\tt \# Alternative: S = [D+1] }  \\
    \Output{Hypersurface set $S$}
    \caption{Hypersurface search (HyperSearch)}
    \label{alg: HyperSearch}
\end{algorithm}

Once $S^*$ has been selected, add the next highest frequency eigenfunction to it to form the set of coordinates $S^* \cup {k}$, but we need to select one more column $k$ of $\tilde{\psi}^{n,\epsilon}$ so that the $(D+1)$-dimensional volume is small. Thus, we seek $k\in[m]\backslash S^*$ such that  $R_{-\zeta}(S^*\cup\{k\})$ is \emph{minimized}. Note that we change the sign of the penalty term in $R$ at the minimization step to penalize for selecting oscillatory eigenvectors as we did in the maximization step.
We term this addition to IES as Hypersurface Search (HyperSearch) and summarize it in Algorithm \ref{alg: HyperSearch}.

\begin{algorithm}
    \SetKwInOut{Input}{Input}
    \SetKwInOut{Output}{Return}
    \SetKwComment{Comment}{$\triangleright $\ }{}
 
    \Input{Featurized data $\mathcal{X}_{\mathcal{F}}$, kernel bandwidth $\epsilon > 0$, number of eigenvectors $m$, emebdding dimension $D+1$, regularization parameter $\zeta$.}

    $\tilde{\psi}^{\epsilon,n}, L^{n,\epsilon} \gets $ diffusion map $(\mathcal{X}_{\mathcal{F}}, \epsilon, m)$ \\
    $\{U(p_i)\} \gets $ RMetric $(\tilde{\psi}^{\epsilon,n}, L^{n,\epsilon}, D+1)$ \\
    $S_* \gets $ IES$(\{U(p_i)[:, [D]]\}, \zeta)$ \\
    $S \gets $ HyperSearch$(\{U(p_i)\}, S_*, \zeta)$ \\
    $\psi_{\mathcal{F}} \gets \psi^{n,\epsilon}[:, S]$ \\
    \Output{Hypersurface set $S$, Independent coordinates $S_*$, Hypersurface embedding $\psi_{\mathcal{F}}$}
    \caption{Learning the residence manifold}
    \label{alg: ResManLearn}
\end{algorithm}

\textbf{Selecting $D+1$ and $\mathcal{F}$}. We summarize the entire pipeline for learning the residence manifold in Algorithm \ref{alg: ResManLearn}. The resulting samples $\{\psi_{\mathcal{F}}(p_i)\}$ are assumed to lie on the surrogate manifold $\widehat{\mathcal{M}}$. Clearly, the validity of this algorithm hinges on an appropriate choice of $D+1$. Ideally, we would like $D+1 = \text{dim}(\mathcal{M}_{\mathcal{F}}) + 1$, so we may estimate the intrinsic dimension of the featurized pointcloud $\mathcal{X}_{\mathcal{F}}$. Alternately, we can consider the $D+1$ which provides the best hypersurface immersion for the given data. This can be achieved as follows: we iterate through the values $2 \leq D+1 \leq m$ and find the $D+1$ for which the difference between the $D$-dimensional volume and the $D+1$ dimensional volume is the largest. This amounts to calculating
\begin{align}
    {\sf HyperSurface}(D+1) := \frac{R_{\zeta}(U'_{S^*}) - \text{R}_{-\zeta}(U_{S})}{|{R}_{\zeta}(U'_{S^*})|}. \label{eq: hypersurface criteria}
\end{align}

The $D+1$ that \emph{maximizes} ${\sf HyperSurface}(D+1)$ should be the dimension which provides the best codimension-1 embedding for the data. For simplicity, here we take $S^* = [D]$ and $S = [D+1]$, motivated by the fact that we are not interested necessarily in independent coordinates but simply by those which form an appropriate hypersurface, i.e minimizing \eqref{eq: hypersurface criteria}. In Figure \ref{fig: HyperSurface}, we plot ${\sf HyperSurface}(D+1)$ for the feature maps given in Section \eqref{subec: group invariant features}.
\begin{figure}[h]
    \centering
    \includegraphics[width=\textwidth]{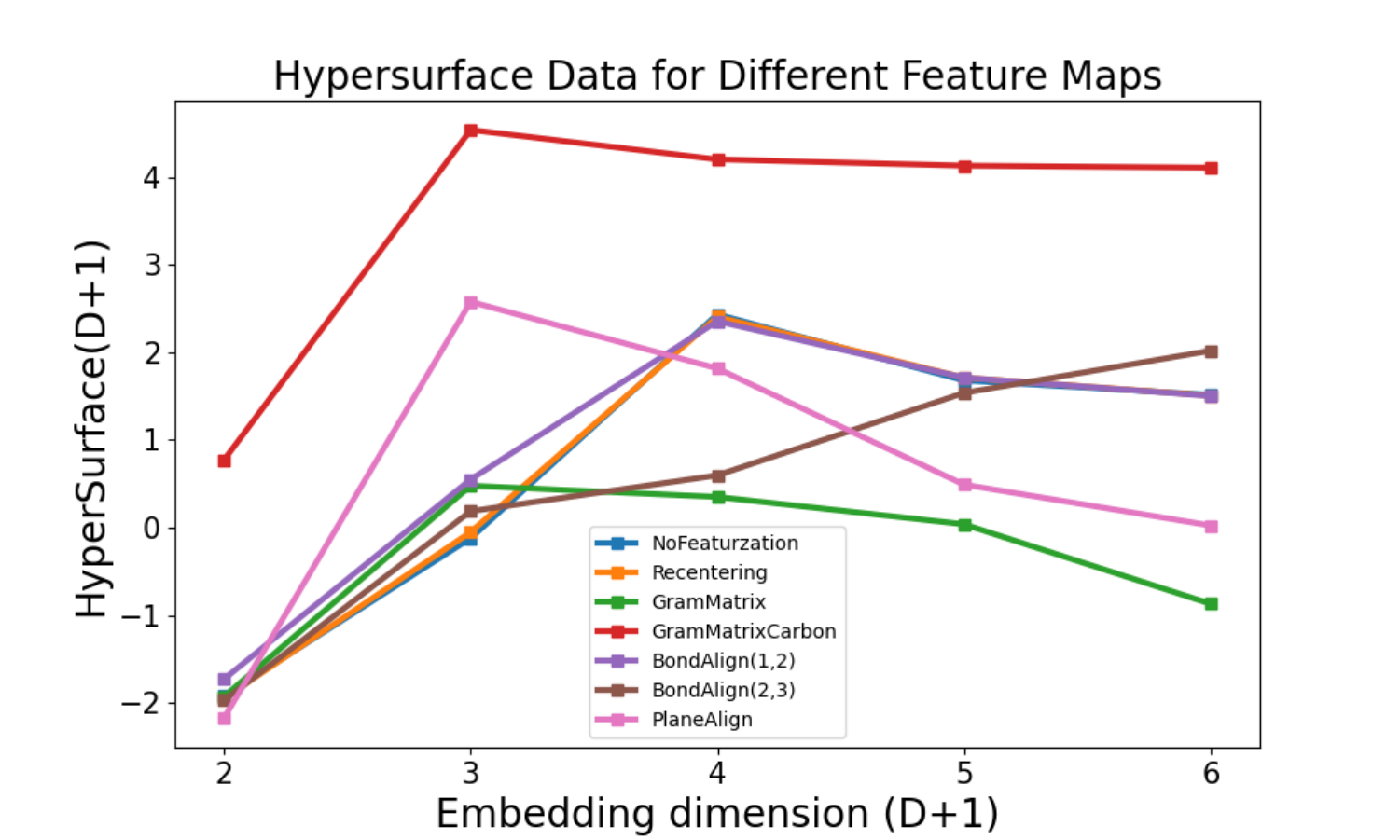}
    \caption{Plotting $\text{HyperSurface}(D+1)$ against $D+1$ for different feature maps. Notably, {\sf GramMatrixCarbon} gives the largest value of HyperSurface($\cdot$) at $D + 1 = 2$ which corresponds to the residence manifold of dimension $D=1$. Moreover {\sf PlaneAlign} seemingly gives the best hypersurface in 3D}.
    \label{fig: HyperSurface}
\end{figure}

\subsection{Data generation and results on selecting diffusion coordinates} To use Algorithm \ref{alg: ResManLearn} we first generated a trajectory from an unbiased Langevin dynamics MD simulation of butane with friction coefficient of $10$ ps$^{-1}$ at 300K. We collected $10^5$ points, each sampled at an interval of $2$ femotseconds. To then determine the best feature map, we subsampled this trajectory at every $100$th point to obtain a point cloud of size $10^3$. This was our raw data $\mathcal{X}$ which we then processed through the different feature maps to obtain the featurized point clouds $\mathcal{X}_{\mathcal{F}}$. The parameter $\epsilon$ was tuned via the K-sum test in \cite{evans2022computing}. On each featurized point cloud we ran Algorithm \ref{alg: ResManLearn}, where the number of eigenvectors was chosen to be $m = 25$ and $\zeta$ was fixed at 0.01 to penalize the frequencies of the eigenfunctions. Moreover, the parameter $\epsilon$ was tuned according to the kernel sum test (see \cite{evans2022computing}). Figure \ref{fig: HyperSurface} reveals that, with the {\sf GramMatrixCarbon} feature map, the best hypersurface is obtained when embedding in two dimensions. Additionally, when using {\sf Recentering} or {\sf NoFeaturization}, we require four dimensions to best embed the data with co-dimension 1. We can also flip this analysis and ask: for a given $D+1$, which feature map $\mathcal{F}$ best embeds the data as a hypersurface? For instance, for both $D+1= 2,3$, the answer is {\sf GramMatrixCarbon}. For $D+1 = 3$, the next best feature map is {\sf PlaneAlign}. A visual inspection of the 3D coordinates why this is the case: {\sf GramMatrixCarbon} embeds the data as a parabolic sheet (bottom row, fourth column, Figure \ref{fig: butane man learning}) while {\sf PlaneAlign} embeds as a circular band in 3D (third column, bottom row in Figure \ref{fig: butane man learning}). However, both can also be embedded in 2D as parabola and a circle respectively (see Figure \ref{fig: 2D featmaps}). For {\sf GramMatrixCarbon}, this is observation consistent with the HyperSurface score, where the feature map attains a high score for both $D+1=2,3$. However, in $D+1=2$, {\sf PlaneAlign} visually embeds as a hypersurface in 2D but \emph{does not} attain a large HyperSurface score. In fact, {\sf PlaneAlign} provides a better embedding than {\sf GramMatrixCarbon} because the {\sf PlaneAlign} embedding is a circle parameterized by the \emph{dihedral angle} while the GramMatrixCarbon embedding is a \emph{parabola} parameterized by $\cos \theta$. This analysis reveals the inherent limitations of HyperSurface--it is only a heuristic and can identify critical dimensions for a feature map. However, it can lead to \emph{false negatives}, as in the case of {\sf PlaneAlign} which attains the lowest score for $D+1=2$, but is clearly visually the ``best" 2D embedding. 

\begin{figure}[h]
    \centering
    \includegraphics[width=0.8\textwidth]{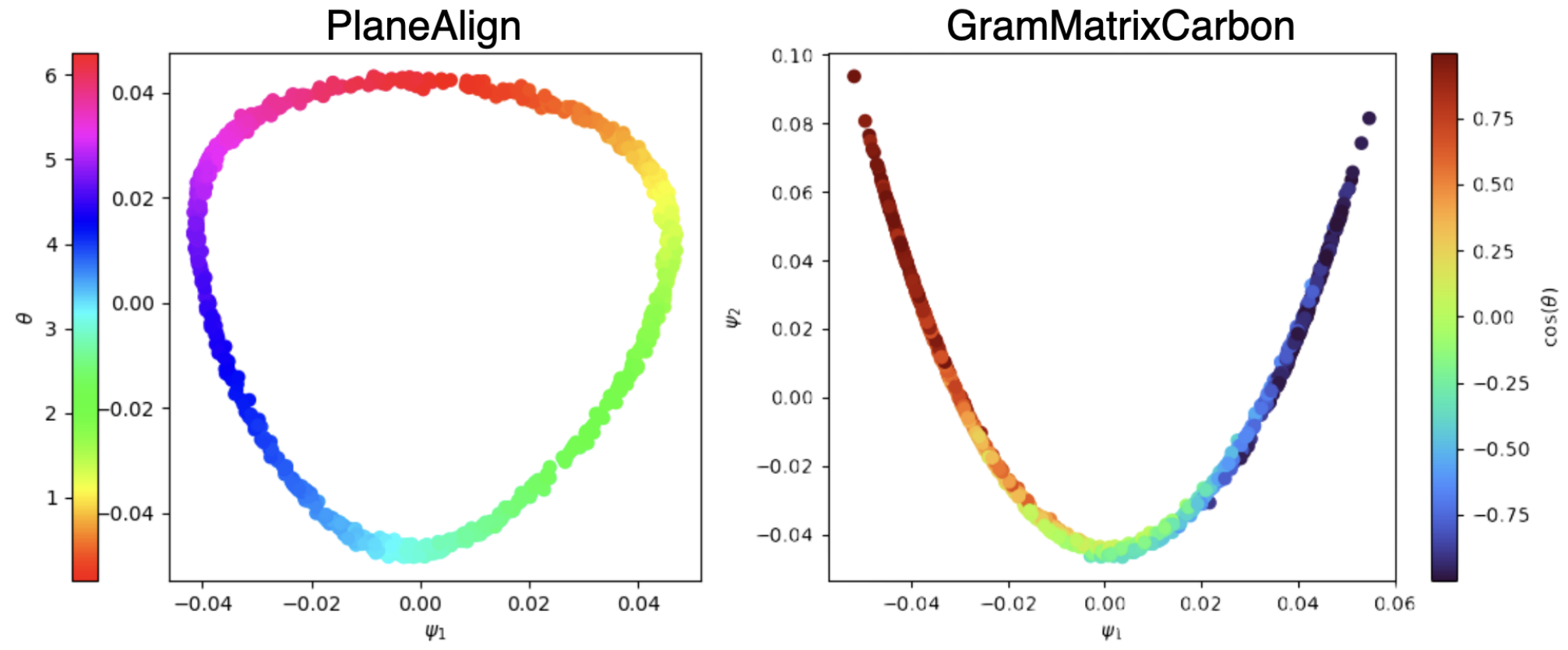}
    \caption{Visualizing the first two selected Laplacian eigenvectors for {\sf PlaneAlign} and {\sf GramMatrixCarbon}. Both embeddings form hypersurfaces in two dimensions.}
    \label{fig: 2D featmaps}
\end{figure}

For $D+1=3$, after {\sf PlaneAlign}, {\sf GramMatrixCarbon}, {\sf BondAlign(1,2)} and {\sf GramMatrix} attain the next top 3 scores respectively. The embeddings into 3D are shown in Figure \ref{fig: butane man learning}. Rather intriguingly, {\sf GramMatrixCarbon} is both embedded as a parabola in two dimensions and as a parabolic sheet in three dimensions. Moreover, the {\sf BondAlign(1,2}) embedding is a paraboloid in 2D (second row, second column) and the {\sf GramMatrix} embedding seems to resemble a wedge product of circles (first row, fourth column). 

\textbf{Importance of hydrogen atoms} 
A highly significant aspect of this analysis is the importance of the hydrogen atoms--the embedding obtained via {\sf GramMatrixCarbon} is parameterized by the $\cos \theta$, which is a somewhat suboptimal collective variable in reproducing transition rates (Table \ref{tab: RatesTheta}). However, in {\sf GramMatrix} when we also include the pairwise inner products with hydrogen atoms, we find that the embedding is smoothly parameterized by $\theta$, where the two disjoint gauche states can be separated. 

\begin{figure}[h]

    \includegraphics[width=\linewidth]{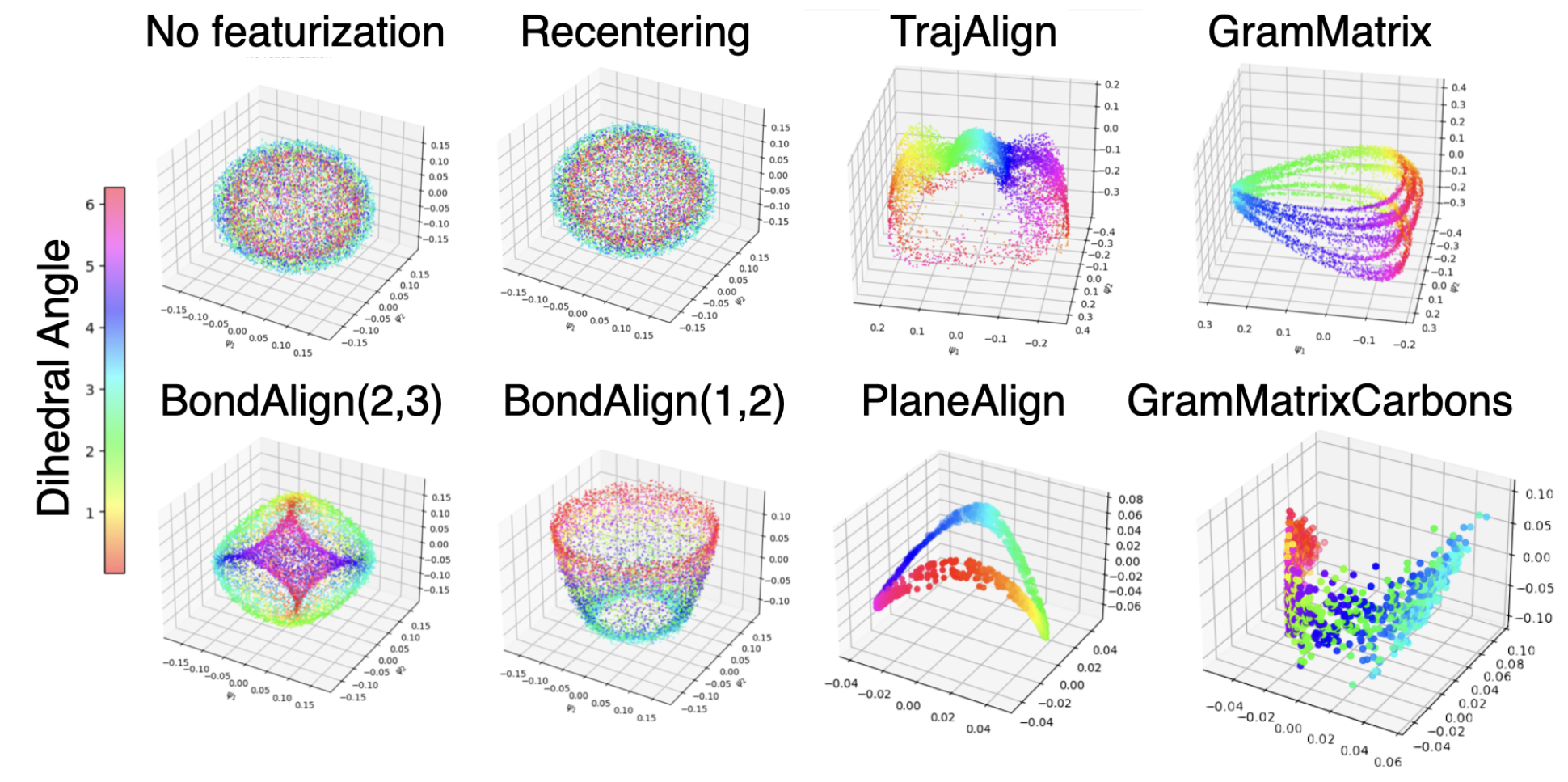}
    \caption{Left to right, top to bottom: The diffusion map embeddings $\{\psi(\mathcal{F}_k(x_i)\}_{x_ii \in \mathcal{X}}$} for eight different choices of feature maps with varying degrees of invariance. The points are coloured according to the dihedral angle. 
    \label{fig: butane man learning}
\end{figure}

\subsection{Diffusion net} After computing the embedding $\psi_{\mathcal{F}}$, we consider the elements of the pointcloud $\{\psi_{\mathcal{F}}(p_i)\}$ as the evaluations of a diffusion net $\Psi_{W_1}: \mathbb{R}^{N} \to \mathbb{R}^{D+1}$ where $\Psi_{W_1} = \Psi^{\mathcal{F}}_{W_1} \circ \mathcal{F}$ as a neural network with parameters $W_1$ (the subscript 1 corresponding to Step 1 of Algorithm \eqref{alg: learning cv}). We train $\Psi^{\mathcal{F}}$ such that (1) $\Psi^{\mathcal{F}}$ matches $\psi_{\mathcal{F}}$ at every point $p_i = \mathcal{F}(x_i)$ and (2) $[\Psi^{\mathcal{F}}_{W_1}]_{j}$, the $j$th coordinate of $\Psi^{\mathcal{F}}$ is a Laplacian eigenfunction with eigenvalue $\lambda_j$. The training objective (1) can be attained by simple mean squared error loss. To achieve training objective (2), we abuse notation and denote $[\Psi^{\mathcal{F}}_{W_1}]_j$ as the evaluations of $[\Psi^{\mathcal{F}}_{W_1}]_j$ on all points $p_i$. Then the vector $L_{n,\epsilon}[\Psi^{\mathcal{F}}_{W_1}]_j$ represents the action of the Laplacian on this function. Objective (2) stipulates that $L^{n,\epsilon}[\Psi^{\mathcal{F}}_{W_1}]_j(p_i) \approx \lambda^{n,\epsilon}_j[\Psi^{\mathcal{F}}_{W_1}]_j(p_i)$. This reasoning leads to the following loss function proposed in \cite{mishne2019diffusion}: 
\begin{align}
    \mathcal{L}_{\rm DNet}(W_1) := \frac{1}{n}\sum_{i=1}^{n}\|\Psi^{\mathcal{F}}_{W_1}(x_i) - \psi^{\mathcal{F}}(p_i)\|_{2}^{2} + \alpha_{\rm DNet}\frac{1}{n}\sum_{i=1}^{n}\sum_{j=1}^{D+1}|L^{n,\epsilon}[\Psi^{\mathcal{F}}_{W_1}]_j(p_i) - \lambda^{n,\epsilon}_{j}[\Psi^{\mathcal{F}}_{W_1}]_j|^2\label{eq: DNet loss}
\end{align}
We can also train a decoder $\Psi_{\rm Dec}$ but since $\mathcal{F}$ is not injective, we can only carry $\Psi(x_i)$ back to $\mathcal{F}(x_i)$. In this case, the decoder is trained with the following reconstruction loss: 
\begin{align}
    L^{\rm Dec}_{\rm DNet}[W^{\rm Dec}_{1},W_{1}] := \frac{1}{n}\sum_{i=1}^{n}\|\Psi^{\rm Dec} \circ \Psi (x_i) - \mathcal{F}(x_i)\|_{2}^{2}. \label{eq: DNet Featurized reconstruction loss}
\end{align} 
Here our focus remains on the encoder $\Psi$. 

\subsection{Undoing spurious topological features.} Among the SE(3)-invariant feature maps, only {\sf BondAlign(2,3)} presents a manifold that cannot be embedded as a hypersurface in either two or threedimensions (Figure \ref{fig: butane man learning}, second row, first column). However, the dihedral angle still seems to smoothly parameterize points on the embedding, which resembles a sphere with handles. We find that this phenomenon persists when applying diffusion maps to data pre-processed via {\sf BondAlign(2,3)} only on the carbon atoms (see top left, Fig. \ref{fig: butane slice}). A further look by slicing the embedding at various $z$-coordinates at this embedding reveals \emph{self-intersections} in three dimensions (see top right, Fig \ref{fig: butane slice}). We additionally embed the data post-processed by the feature map  {\sf BondAlign(2,3)} into four dimensions and visualize its sliced data in three dimensions with the third coordinate being the \emph{fourth} Laplacian eigenfunction $\psi_4$ in the bottom left of Fig \ref{fig: butane slice}.  This reveals that the self-intersections are actually spurious and that the three-dimensional sliced data forms \emph{loops} which can be parameterized using the dihedral angle. This suggests that there is a functional dependence among the first four coordinates $\psi = (\psi_1, \ldots, \psi_4)$. We quantify this functional dependence via the following energy proposed in \parencite{kevrekidis2024thinner}: 
\begin{align}
    \mathcal{E}(\Psi) := \sum_{i < j}|\langle \nabla_{\mathbb{R}^{d_{\mathcal{F}}}}\Psi_i, \nabla_{\mathbb{R}^{d_{\mathcal{F}}}}\Psi_j\rangle|^2. \label{eq: conformal autoencoder energy}
\end{align}
The functional \eqref{eq: conformal autoencoder energy} was suggested as a regularization term in a \emph{conformal autoencoder}. We now compute a map $\Psi: \mathbb{R}^{14 \times 3} \to \mathbb{R}^{D+1}$ as a neural network trained to be a conformal autoencoder whose coordinates are Laplacian eigenfunctions. In particular, the encoder is trained to minimize the following loss function: 
\begin{align}
    L^{\rm Enc}_{\rm LAPCAE}[W_1] := \frac{1}{n}\sum_{i=1}^{n}\sum_{j=1}^{D+1}|L^{n,\epsilon}\Psi^{\mathcal{F}}_{W_1}(p_i) - \lambda^{n,\epsilon}_{j}\Psi^{\mathcal{F}}_{W_1}(p_i)|^2 + \alpha_{\rm LAPCAE}\mathcal{E}(\Psi_{W_1}). \label{eq: LAPCAE encoder loss}
\end{align}
We train the overall autoencoder $\Psi^{\rm Dec} \circ \Psi$ with a weighted combination of $L^{\rm Dec}_{\rm DNet}$ \eqref{eq: DNet Featurized reconstruction loss} and $L^{\rm Enc}_{\rm LAPCAE}$ \eqref{eq: LAPCAE encoder loss}: 
\begin{align}
    L_{\rm LAPCAE}[W^{\rm Dec}_{1}, W_1] := L^{\rm Dec}_{\rm DNet}[W^{\rm Dec}_{1}] + \alpha^{\rm Enc}_{\rm LAPCAE}L^{\rm Enc}_{\rm LAPCAE}[W_1]. \label{eq: LAPCAE loss}
\end{align}
After training $\Psi$ on \eqref{eq: LAPCAE loss} we project the all-atom data to obtain a low-dimensional embedding formed by points $\{\Psi^{\mathcal{F}}(p_i)\}_{i=1}^{n}$ (bottom left, Fig \eqref{fig: butane slice}). Remarkably, the LAPCAE removes the self-intersections, resulting in a cylinder embedded in three dimensions. Moreover, the points are smoothly organized by the dihedral angle. Undoing topological features like spurious knots is an important challenge in data visualization and manifold learning.  Our result is therefore of independent interest to geometric data science in addition to the MD community. 
\begin{figure}[h]
    \centering
    \includegraphics[width=\linewidth]{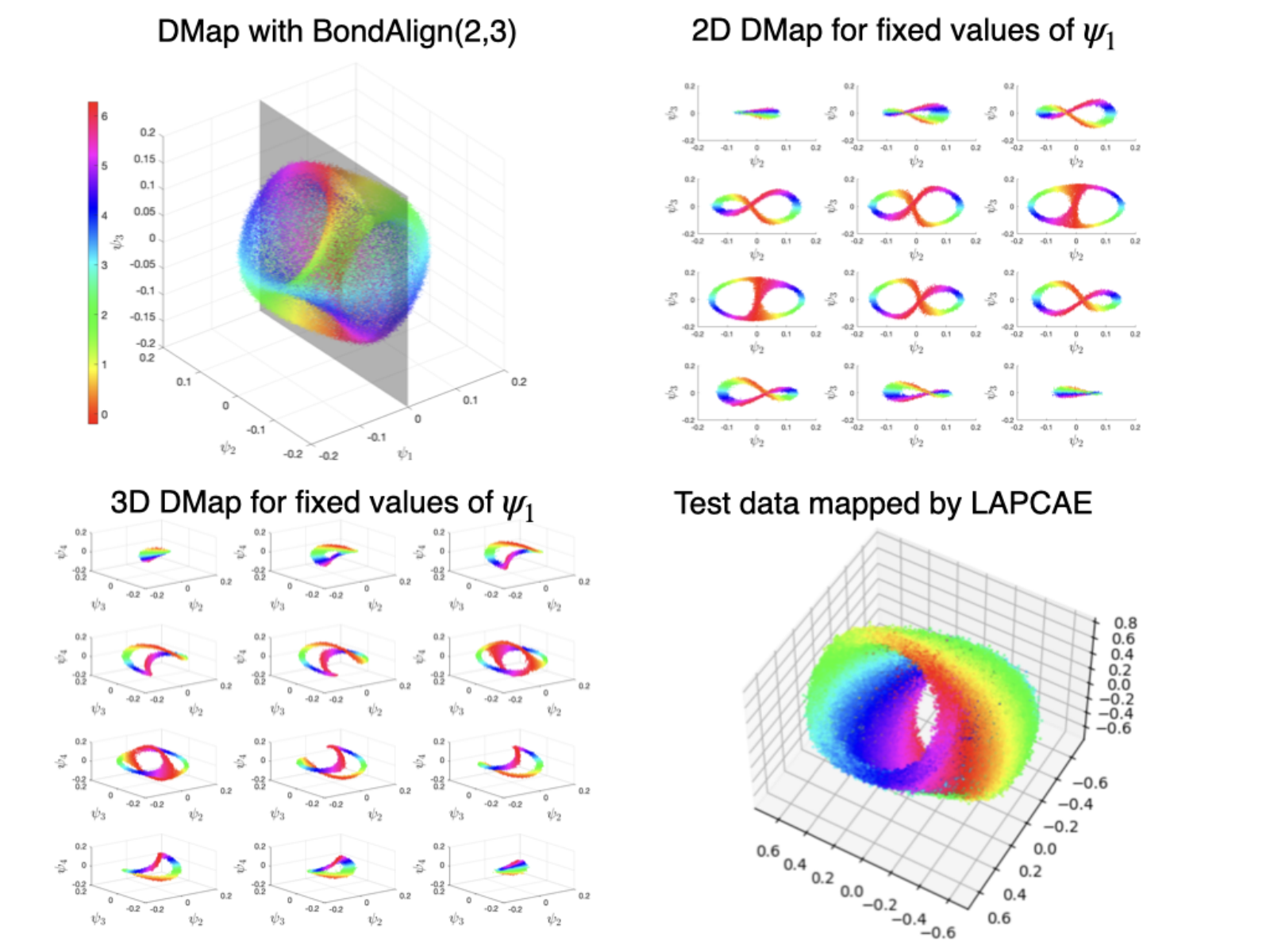}
    \caption{Top left: The diffusion map embedding after featurization by recentering the $C_2-C_3$ bond on the $x$-axis. Top right: Slices of the embedding by planes parallel to the $xy$-plane. Bottom left: Visualizing the sliced data along with the corresponding fourth eigenfunction. Bottom right: The corrected embedding after mapping via a Laplacian conformal autoencoder.}
    \label{fig: butane slice}
\end{figure}

\subsection{Computational details} After acquiring hypersurface coordinates for each feature map, we took every $10$th point of our overall trajectory $\{X_{i\Delta t}\}_{i=1}^{10^5}$ as a training set $\mathcal{X}$ of size $10^4$. The rest of the points were considered part of the testing set. We then postprocessed this point cloud $\mathcal{X}$ using the group invariant feature maps described in Section \eqref{subec: group invariant features} to obtain featurized pointclouds $\mathcal{X}_{\mathcal{F}}$ and computed the diffusion map on the selected hypersurface coordinates. We then trained a diffusion net with encoder $\Psi$ and decoder $\Psi^{\rm Dec}$. Both encoders and decoders were four-layer fully connected networks with either $\tanh$ or $x + \sin^2(x)$ activation. We trained these using an ADAM optimizer with learning rate $0.01$. For training LAPCAE with the feature map {\sf BondAlign} applied to the carbon atoms, the constant $\alpha_{\rm DNet}, \alpha^{\rm Enc}_{\rm LAPCAE}$ were chosen to be $1.0$. We provide additional details regarding neural network training in Appendix \ref{app: Computational details}.  

\section{Learning CVs from the residence manifold}
\label{sec: learning cvs}
\subsection{Learning the confining potential $V_1$} 
We assume now that the diffusion net $\Psi$ has been trained and that the data $\widehat{\mathcal{X}} = \{y_i\}$ lies on the surrogate manifold $\widehat{\mathcal{M}}$ where $y_i = \Psi(x_i)$. Now we aim to learn a surrogate potential $\widehat{\Phi}$ which vanishes exactly on $\widehat{\mathcal{M}}$. We are aided by the fact that $\widehat{\mathcal{M}}$ can be presented as a hypersurface. In this case, a loss for learning $\widehat{\Phi}$ may be derived using the eikonal equation. In particular, $\widehat{\mathcal{M}}$ may be exhibited (possibly locally) as the boundary of some open set in $\mathbb{R}^{D+1}$. In this case, then the signed distance function (SDF) $\widehat{\Phi}_{\rm SDF}$ to this open set satisfies $\|\nabla \widehat{\Phi}_{\rm SDF}\|^2 =1$. We propose that this SDF be a representation of the desired surrogate potential $\widehat{\Phi}$. This is because $\nabla \widehat{\Phi}_{\rm SDF}(y) = -N(y)$ where $N(y)$ is the outward pointing normal vector to $\widehat{\mathcal{M}}$ at $y$, i.e the condition we want the surrogate potential $\widehat{\Phi}$ to satisfy. A loss function for $\Phi_{W_2}$ is given by:
\begin{align}
    L_{\rm eikonal}[W_2] :=\frac{1}{n}\sum_{i=1}^{n}(\|\nabla \Phi_{W_2}(y_i)\|_{2}^{2}-1)^2 + \alpha_{\rm zero} \|\Phi_{W_2}(y_i)\|_{2}^{2}. \label{eq: eikonal} 
\end{align}
As an alternate approach, given the point cloud $\widehat{\mathcal{X}}$, the normals at each point $N(y_i)$ may be estimated directly without needing a neural network by taking a neighborhood of $y_i$ and fitting a tangent plane to this neighborhood. The estimated normal $n_i$ at $y_i$ is then the normal vector to this tangent plane. This method is summarized in algorithm \eqref{alg: pointcloud normals}. 

\begin{algorithm}[H]
\label{alg: pointcloud normals}
\DontPrintSemicolon
\KwIn{Point cloud $\widehat{\mathcal{X}} = \{y_i\}_{i=1}^N \subseteq \mathbb{R}^{D+1}$, number of neighbors $k$}
\KwOut{Estimated normals $\{n_i\}_{i=1}^N$}

Build a spatial index (e.g., KD-tree) for fast nearest-neighbor queries\;

\ForEach{$y_i \in \mathcal{P}$}{
    Find the $k$ nearest neighbors of $y_i$: $\{y_j\}_{j=1}^k \leftarrow \texttt{knn}(y_i, k)$\;

    Compute the centroid of neighbors: 
    $$\bar{y} \leftarrow \frac{1}{k} \sum_{j=1}^k y_j$$

    Center the neighbors:
    $$Y \leftarrow [y_j - \bar{y}]_{j=1}^k$$

    Compute the covariance matrix:
    $$C \leftarrow \frac{1}{k} Y^\top Y$$

    Perform eigen-decomposition or SVD of $C$ to get eigenvalues $\lambda_1 \leq \cdots \leq \lambda_d$ and eigenvectors $v_1, \ldots, v_d$\;

    Assign the normal vector corresponding to the eigenvector with the smallest eigenvalue:
    $$n_i \leftarrow v_1$$ 
}

\Return{$\{n_i\}_{i=1}^N$}
\caption{Estimate Normals via k-NN and Tangent Plane Fitting}
\end{algorithm}

After estimating the point cloud normals $\{n_i\} \subseteq \mathbb{R}^{D+1}$ we use these to supplement $L_{\rm eikonal}$ with an additional regularization term to aid learning the surrogate potential $\Phi$. Thus, the final loss for learning $\Phi$ is given by: 
\begin{align}
    L_{\rm potential}[W_2] := L_{\rm eikonal}[W_2] + \frac{\alpha_{\rm normals}}{n}\sum_{i=1}^{n}\|\nabla \Phi_{W_2}(y_i) - n_i\|_{2}^{2}. \label{eq: full loss for learning surrogate potential}
\end{align}
Here $\alpha_{\rm normals} > 0$ is a hyperparameter. In Figure \ref{fig: gradients of surrogate potential} we present $\nabla \widehat{V}$ learned for the {\sf BondAlign(1,2)}, {\sf BondAlign(2,3)}, and {\sf PlaneAlign} feature maps, confirming that the gradients of $\widehat{\Phi}$ are normal to the learned manifold $\widehat{\mathcal{M}}$. 
\begin{figure}[h]
    \centering
    \includegraphics[width=\textwidth]{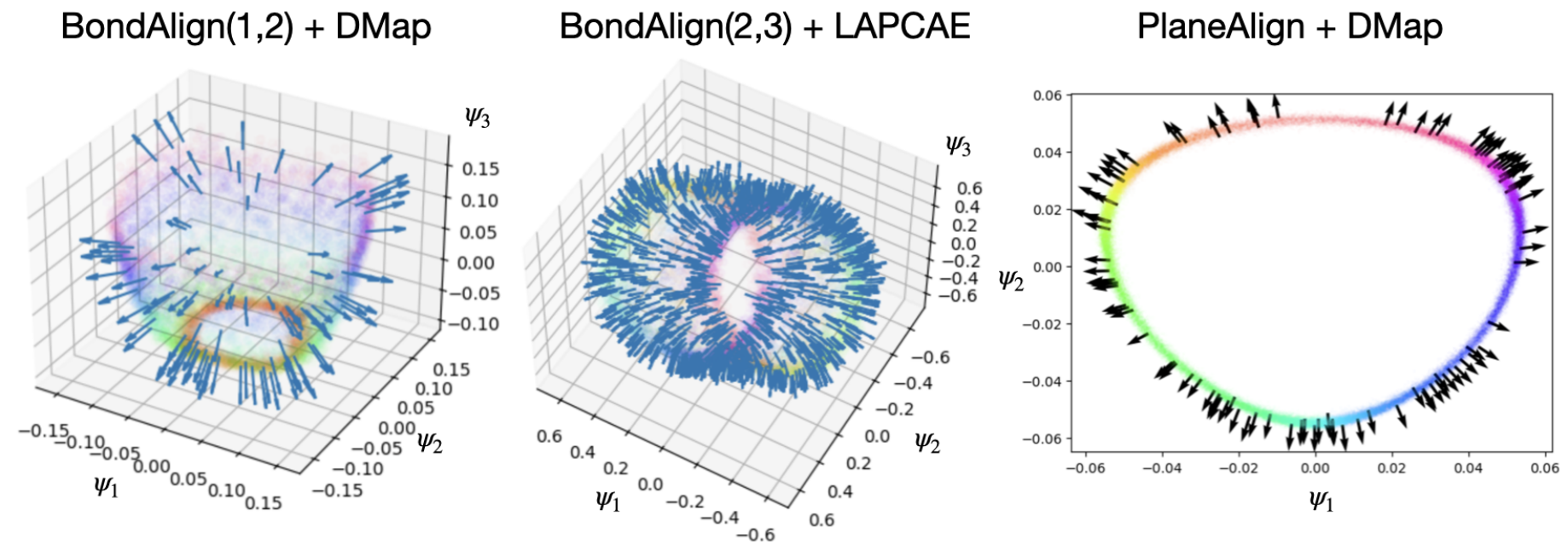}
    \caption{Gradients of potential $\nabla \widehat{\Phi}$ learned for {\sf BondAlign(1,2)}, {\sf BondAlign(2,3)} and {\sf PlaneAlign}. Here $\widehat{\Phi}$ is learned by minimizing $L_{\sf potential}$. }
    \label{fig: gradients of surrogate potential}
\end{figure}

\subsection{Learning collective variable $\xi$} After acquiring the surrogate potential $\widehat{\Phi}$ that vanishes on the surrogate manifold $\widehat{M}$, we finally train the collective variable $\widehat{\xi}_{W_3}$, given by another neural network with parameters $W_3$ trained to satisfy the orthogonality criterion with respect to $\widehat{\Phi}$. We will proceed to learn CVs following Algorithm 1 only for the two most promising featurizations, {\sf BondAlign(2,3)} and {\sf PlaneAlign}. To do so, we consider the surrogate potential $\widehat{\Phi}$ obtained by minimizing $L_{\rm potential}$ in \eqref{eq: full loss for learning surrogate potential}. In this case $\widehat{\Phi} \approx \widehat{\Phi}_{\rm SDF}$.  However, 
\begin{equation}\label{eq:HessPhiGradPhi}H\widehat{\Phi}_{\rm SDF}\nabla \widehat{\Phi}_{\rm SDF} = 0,
\end{equation}
where $H\widehat{\Phi}_{\rm SDF}$ is the Hessian of the signed distance function. Since $\nabla\widehat{\Phi}$ is the normal vector to $\mathcal{M}$, this implies that the rows of $H\widehat{\Phi}$ given by $\nabla \partial_{i}\widehat{\Phi}$ satisfy the orthogonality condition \eqref{eq: orthogonality condition}. Therefore, $\partial_{i}\widehat{\Phi}$ may be directly used as a collective variable; alternately, a separate encoder $\widehat{\xi}$ may be trained such that the \emph{normalized gradients} of $\widehat{\xi}$ given by $\nabla \widehat{\xi}/\|\nabla \widehat{\xi}\|$ are aligned with the normalized $\nabla \partial_{i}\widehat{\Phi}$. One way to avoid dividing by $\|\nabla \widehat{\xi}\|$ is to observe that the Cauchy-Schwarz inequality $|\langle\nabla \widehat{\xi}, \nabla \partial_{i}\widehat{ \Phi}\rangle|^2 \leq \|\nabla \widehat{\xi}\|^2\|\nabla { \partial_{i}}\widehat{\Phi}\|^2$ is an equality if and only if $\nabla \widehat{\xi}$ and  $ \nabla \partial_{i} \widehat{\Phi}$ are parallel. Therefore, a loss function for learning $\widehat{\xi}$ is given by: 
\begin{align}
    L_{\rm alignment}[W_3] = \sum_{i=1}^{n}\|y_i - \xi_{\rm Dec}\circ \xi(y_i)\|^{2}_{2} + \alpha_{\rm OC}(|\langle\nabla \widehat{\xi}, \nabla \partial_{i}\widehat{\Phi}\rangle|^2 - \|\nabla\widehat{\xi}\|^2\| \nabla \partial_{i}\widehat{ \Phi}\|^2)^2. \label{eq: L alignment loss}
\end{align}
In Figure \ref{fig: CV grads}, we visualize the gradients of $\widehat{\xi}$ for the feature maps {\sf BondAlign(2,3)} and {\sf PlaneAlign} which are orthogonal to the surrogate potential $\widehat{\Phi}$ on the learned manifold $\widehat{\mathcal{M}}$. For {\sf BondAlign(2,3)}, we use the loss $L_{\rm alignment}$ \eqref{eq: L alignment loss} with $\partial_1 \widehat{\Phi}$, where $\widehat{\Phi}$ is learned by minimizing $L_{\rm eikonal}$. For {\sf PlaneAlign}, we find that $\widehat{\xi}(x) = {\tt arctan2}(x_2,x_1)$  ({\tt arctan2} returns angle from $-\pi$ to $\pi$) is orthogonal to $\nabla \widehat{\Phi}$ and therefore satisfies \eqref{eq: orthogonality condition}. Thus, given an interpretable enough visualization, CVs may be devised by inspection. Interestingly, the CV learned via PlaneAlign correlates significantly with the dihedral angle.  
\begin{figure}
    \centering
    \includegraphics[width=\textwidth]{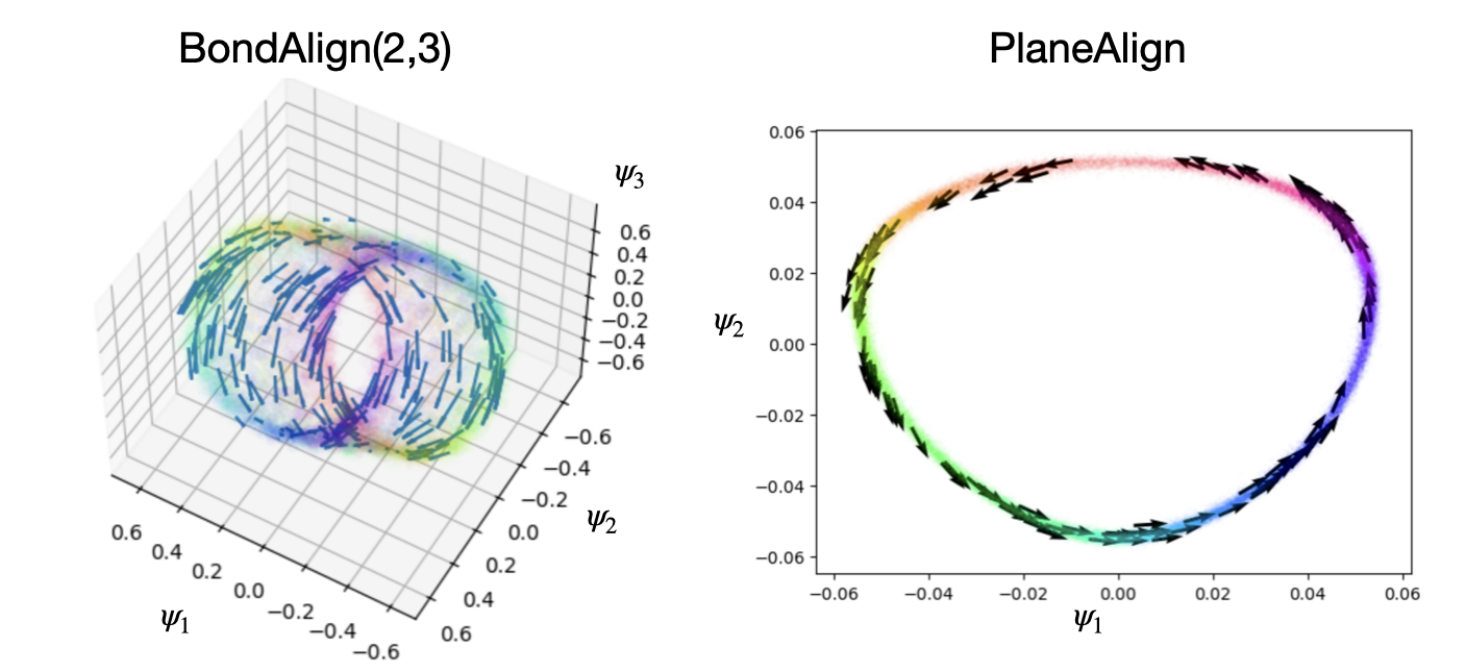}
    \caption{Gradients of CVs for {\sf BondAlign(2,3)} and {\sf PlaneAlign}.}
    \label{fig: CV grads}
\end{figure}
\begin{figure}
    \centering
    \includegraphics[width=0.7\textwidth]{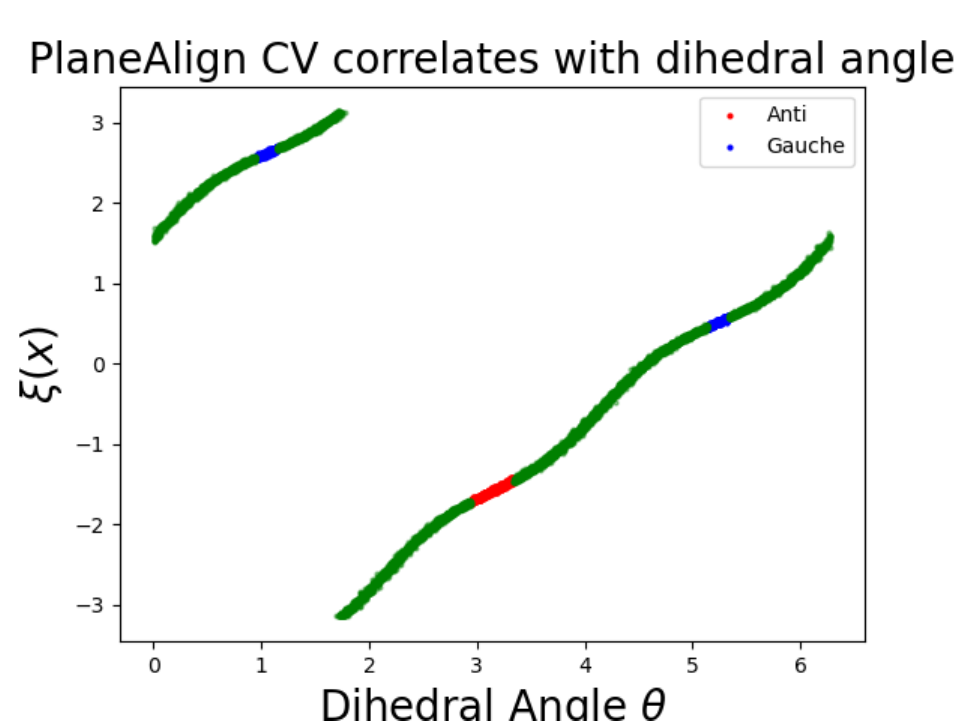}
    \caption{We visualize $(\theta(x_i),\xi(x_i))$ for $x_i \in \mathcal{X}$, the training data used for learning $\xi$. Note that $\xi$ is a nonlinear function of $\theta$--the single discontinuity is due to the use of the {\tt arctan2} function which requires a branch cut at $\theta = 2\pi$.}
    \label{fig: dihedral correlates CV}
\end{figure}
\subsection{Computational details} For both {\sf PlaneAlign} and {\sf BondAlign(2,3)}, $\widehat{\Phi}$ and $\widehat{\xi}$ were taken to be four-layer fully connected networks with $\arctan$ and $x + \sin^2(x)$ activations respectively. The use of such nonstandard activations was motivated by the spherical/toroidal structure of the surrogate manifolds and the modeling of $\widehat{\Phi}$ as a \emph{signed} distance. The exact architectures and training details have been provided in Table \eqref{tab: NN_details} in Appendix \ref{app: Computational details}.

\subsection{Reproducing transition rates}
\label{subsec: reproducing transition rates}
Using CVs $\xi := \widehat{\xi} \circ \Psi$ learned from Algorithm \ref{alg: learning cv} we now reproduce the transition rate $\nu_{AB}$ for the anti-gauche transition from the set 
$$A = \{|\theta-\pi| \leq 0.2\}\quad {\rm to}\quad B = \{|\theta-\pi/3| \leq 0.1\} \sqcup \{|\theta-5\pi/3| \leq 0.1\}.$$ 
We will do so by computing $\tilde{\nu}_{AB}$ via \eqref{eq: transition rate in CVs}. This necessitates computing the free energy $f$ and the diffusion tensor $M(z)$. We compute $f$ via \emph{well-tempered metadynamics} \cite{barducci2008well} and $M$ via the spring force method \cite{maragliano2006string} respectively. 

Our results are presented in Table \ref{tab: transition rates}, where we organize the machine-learned collective variables (ML CVs) by the underlying feature map. We find that although the {\sf BondAlign(2,3)} feature map results in a coherent manifold, the CVs learned using such a manifold do not separate the metastable states $A$ and $B$ (see Figure \ref{fig: bondalign 23 statesep}). However, the CV learned using {\sf PlaneAlign} corresponds nearly exactly to the dihedral angle, and therefore separates metastable states. In CV space these are identified with the following intervals: 
\begin{align}
    A = [-1.75,-1.41], \: B = [0.43, 0.60] \: \sqcup \: [2.54, 2.71].
\end{align}
Using these and applying formula \eqref{eq: transition rate in CVs} led to a transition rate with less than ten percent relative error. The free energy and diffusion tensors used for computing the transition rates are provided in Figure \eqref{fig: PlaneAlign}. For completeness, we visualize the free energy landscape and diffusion tensors for the two-dimensional CV learned using {\sf BondAlign(2,3)} in Figure \ref{fig: BondAlign23} even though we do not compute the anti-gauche transition rate using this CV. Notably the diffusion tensor for the CV learned via PlaneAlign is different than the dihedral angle.
% \usepackage{booktabs}
\begin{table}[h!]
\begin{center}
\begin{tabular}{@{}ll@{}}
\toprule
Collective variable          & Transition rate (ps$^{-1}$)     \\ \midrule
Reference                    & 1.13 $\pm$ 0.08 $\times 10^{-2}$ \\
Dihedral angle $\theta$      & 1.41 $\times 10^{-2}$               \\
$(\sin \theta, \cos \theta)$ & $1.19 \times 10^{-2}$            \\
$\cos \theta$                & $1.52 \times 10^{-2}$ \\
$\xi$, {\sf BondAlign(2,3)}               & Metastable states not separated             \\
$\xi$, {\sf PlaneAlign}                   & $1.25 \times 10^{-2}$              \\ \bottomrule \\
\end{tabular}

\caption{Transition rates for the anti-gauche transition in both classical and machine-learned collective variables. We organize the ML CVs $\xi$ according to the relevant feature map in use.}
\label{tab: transition rates}

\end{center}
\end{table}

\section{Discussion}

In this work, we have made several %key 
contributions to the theory and application of quantitative coarse-graining. 
%In Section \eqref{sec: background} we reviewed the fundamental theory behind quantitative coarse-graining and demonstrated how the orthogonality condition (OC) naturally emerges as a mechanism for reducing error estimates. Furthermore, 
We established the equivalence between the orthogonality condition  \eqref{eq: orthogonality condition}, which naturally emerges as a mechanism for reducing the modeling error measured via the relative entropy, and the Projected Orthogonality Condition \eqref{eq: orthogonality condition 2} (Proposition \eqref{prop: oc poc equivalence}), which ensures that the pathwise distance remains small. Through our study of butane, we also showed that Assumption \ref{ass: xi} demanding that $D\xi D\xi^\top \succcurlyeq \delta_{\xi}^2I_d$ can be relaxed to accommodate rank-deficient but non-degenerate collective variables (CVs), thereby broadening the applicability of these conditions. 

We introduced a framework summarized in Algorithm \ref{alg: learning cv} inspired by the orthogonality condition  \eqref{eq: orthogonality condition}, to construct CVs that flow along the surrogate manifold $\mathcal{M}$. A crucial component of our approach is the residence manifold learning algorithm (Algorithm \ref{alg: ResManLearn}), which carefully integrates several tools from manifold learning to construct a hypersurface representation. This structured approach facilitates an efficient computation of the normal vector, allowing for the direct construction of a CV whose gradients lie in the tangent space. 

{\bf The significance of the feature map.} The key emergent theme from this approach is the extreme importance of choosing the correct feature map for post-processing the data. The existence of several plausible but structurally different surrogate manifolds suggests that a more intricate design of feature maps beyond standard techniques such as RMSD alignment or pairwise distances can provide significant accelerations in downstream tasks such as transition rate estimation conducted in Section \eqref{subsec: reproducing transition rates}. In particular, our hand-crafted {\sf PlaneAlign} feature map causes the diffusion net embedded to be parameterized by the dihedral angle $\theta$. This parameterization is duly recovered by the collective variable $\xi$ leading to a successful reproduction of the anti-gauche transition rate (see Table \eqref{tab: transition rates}). Furthermore, we provided a heuristic function {\sf HyperSurface}, \eqref{eq: hypersurface criteria}, which combines Laplacian coordinates to embed diffusion map manifolds as hypersurfaces irrespective of the choice of feature map (see Figures \eqref{fig: 2D featmaps}, and \eqref{fig: butane man learning}). However, {\sf HyperSurface} remains purely a heuristic because of (1) false negatives (for instance, {\sf PlaneAlign} attains a low 2D score but is a good 2D embedding) and (2) self-intersections despite functional independence. We resolved the second issue by simultaneously promoting functional independence in Laplacian eigenfunctions through a neural network loss \eqref{eq: LAPCAE loss} leading to LAPCAE, enabling a well-structured embedding of the surrogate manifold \eqref{fig: butane slice}.

{\bf Limitations.} Despite the above advances, several limitations remain. Our current implementation has only been applied to the butane molecule, a very simple MD system. Moreover, in our framework, the orthogonality condition \eqref{eq: orthogonality condition} is imposed on the surrogate manifold rather than directly in the all-atom space \( \mathbb{R}^{N} \).  While it is theoretically possible to define an orthogonality condition in \( \mathbb{R}^{N} \) via a parameterization \( \xi = \widehat{\xi} \circ \Psi \) and a function \( \Phi = \widehat{\Phi} \circ \Psi \), we found this approach impractical. The presence of numerous ways to satisfy the constraint \( \nabla \xi \cdot \nabla \Phi = 0 \) in \( \mathbb{R}^{N} \) introduces excessive local minima in the loss function, making the learning of \( \xi \) less robust.

{\bf Can the surrogate manifold be a CV?} Given the ability of diffusion maps to identify the slow manifold of the overdamped Langevin dynamics \eqref{eq: OLD}, the learned manifold coordinates \( \Psi \) could, in principle, serve as CVs. Similar ideas have been implemented in \cite{rydzewski2024spectral} or \cite{smith2018multi} where CVs are directly designed to optimize spectral criteria. Here we use eigenfunctions to reduce the problem to a surrogate space $\mathbb{R}^{D+1}$ and then apply the relevant CV criteria in this surrogate space. We emphasize that enforcing the orthogonality condition \eqref{eq: orthogonality condition} remains beneficial despite the complexity reduction afforded by the map $\Psi$. In particular, the manifold coordinates \( \Psi \) may still be high-dimensional depending on the choice of embedding dimension \( D+1 \), making direct usage impractical. Moreover, the orthogonality condition \eqref{eq: orthogonality condition} provides an implicit way to encode functional dependence on the intrinsic parameterization of the surrogate manifold without requiring a coordinate chart in \( \mathbb{R}^{D+1} \). Specifically, locally the surrogate manifold has $D+1$ degrees of freedom and can be represented as the graph of a function. However, recovering such a function for every point is highly cumbersome and the orthogonality condition \eqref{eq: orthogonality condition} provides a way to implicitly define $\xi$ on these intrinsic degrees of freedom. 

{\bf Future applications and implementational improvements.} The above insights and limitations open numerous directions of future work. One primary direction is to apply our methodology towards studying the conformational dynamics of synthetic peptide chains such as chignolin \cite{harada2011exploring} or AIB9, and to the resolution of more advanced challenges such as cryptic pocket discovery or \cite{wang2022data} nucleation \cite{wang2024local}. Given the challenges of defining the orthogonality condition directly in \( \mathbb{R}^{N} \), future work could explore alternative formulations that mitigate the issue of excessive local minima in the loss function. This may involve developing regularization techniques or optimizing over constrained subspaces that better capture physically meaningful CVs.

{\bf Open theoretical questions.} In addition to applications in molecular dynamics, here we also underline some open questions in quantitative coarse-graining and manifold learning prompted by our approach. With regards to quantitative coarse-graining, in Section \eqref{sec: effective dynamics analysis}, we have demonstrated empirically that the diffusion tensor $M(z)$ need not be full rank to reproduce transition rates. This assumption is ubiquitous in most recent works on quantitative coarse-graining, thus leaving an intriguing avenue to extend the existing theory toward error estimates for degenerate diffusions. Moreover, with regards to manifold learning, we make use of group invariant diffusion maps via a group invariant feature map in Algorithm \ref{alg: ResManLearn}, thus furthering the project started in \cite{rosen2023g, hoyos2023diffusion} of encoding group structure into spectral embeddings. These works also use a group invariant kernel to perform diffusion maps but do so by imparting group invariance via averaging against the Haar measure of the underlying Lie group. This necessitates generating the orbit of every element in the input point cloud. Here we have presented an alternative approach that does not require these orbits by using a feature map before applying the kernel. A comparison and possible integration of both approaches is therefore interesting and may hold potential for accelerating the accuracy and expressibility of diffusion maps for MD data.  

\section{Conclusion}

Our work advances the theory and practice of quantitative coarse-graining by elucidating the role of the orthogonality condition in reducing modeling error and preserving transition rates in overdamped Langevin dynamics. We introduced a learning framework that constructs collective variables by imposing the orthogonality condition on a surrogate space defined via manifold-based representations. We demonstrated that group invariant featurization is critical for defining such representations and therefore heavily affects coarse-graining performance. Through our butane case study, we demonstrated the role of analytical conditions from quantitative coarse-graining and provided practical heuristics, such as {\sf HyperSurface} and LAPCAE, to address manifold learning challenges in our proposed algorithm. While our approach currently remains limited to simple systems and surrogate spaces, it opens up rich avenues for applying these techniques to more complex biomolecular dynamics and exploring fundamental questions in coarse-graining and manifold learning.

 \section{Acknowledgements}
 This work was partially supported by AFOSR MURI grant FA9550-20-1-0397 and by 
  NSF REU grant DMS-2149913. The first author would like to thank Akash Aranganathan and Dr. Eric Beyerle for helpful discussions regarding MD simulations in the high-friction regime.

 \appendix
\setcounter{equation}{0}
\renewcommand{\theequation}{\Alph{section}-\arabic{equation}}
    \setcounter{lemma}{0}
    \renewcommand{\thelemma}{\Alph{section}\arabic{lemma}}

\section*{Appendix}

\section{Relation to full Langevin dynamics and reaction rate scaling} 
\label{app: A}

Typically, MD simulations track not just the positions $X_t$ but also the velocities of atoms $v_t$ (the lower case notation has been purposely chosen to not confuse with the potential $V$) under the \emph{Langevin dynamics} given by
\begin{align}
    dX_t &= v_t\,dt, \label{eq: full Langevin 1} \\
    m\,dv_t &= -\gamma m v_t\,dt {-} \nabla V(X_t)\,dt + \sqrt{{ 2}k_{B}T\gamma}\:m^{1/2}dW_t. \label{eq: full Langevin 2}
\end{align}
Here $\gamma$ is the friction coefficient, $T$ is the temperature, $m$ is the diagonal matrix of masses of individual particles, and $k_B$ is Boltzmann's constant used to make terms dimensionless. In the regime $\gamma \gg 1$, the Langevin dynamics \eqref{eq: full Langevin 1}, \eqref{eq: full Langevin 2} may be approximated by the following anisotropic overdamped Langevin dynamics with $\beta^{-1} = k_BT $: 
\begin{align}
    dX_t &= {-} m^{-1}\nabla V(X_t)\gamma^{-1}\,dt + \sqrt{2\beta^{-1} \gamma^{-1}}m^{-1/2}dW_t. 
\end{align}
In this case, $\mu \propto \exp(-\beta V)$ is still the invariant measure. To remove the implicit effect of friction, we must rescale time as $\tau = \gamma^{-1}t$. This results in a time-rescaled overdamped Langevin dynamics given by
\begin{align}
    dX_{\tau} &= {-}m^{-1}\nabla V(X_{\tau})\,d\tau + \sqrt{{ 2}\beta^{-1}}m^{-1/2}dW_{\tau}. \label{eq: rescaled time OLD}
\end{align}
Under the overdamped approximation, the dynamics are anisotropic due to the involvement of the mass matrix $m$. In this case, the diffusion tensor must feature this mass rescaling \cite[Eq. 5]{cameron2013estimation} and thus be given by: 
\begin{align}
    M(z) := \mathbb{E}_{\mu}[(D\xi) m^{-1}(D\xi)^{\top} \mid \xi = z].
\end{align}
With this diffusion tensor, the rate formula \eqref{eq: transition rate in CVs} applies for computing the transition rate $\nu'_{AB}$. However, this transition rate is for transitions occurring in the timescale $\tau$; thus to recover the timescale $t$, we must rescale by the relevant factor of $\gamma$. This is pertinent when computing transition rates. For example, as a heuristic, the transition rate $\nu'_{AB}$ computed for the dynamics \eqref{eq: rescaled time OLD} would be under the units $1/\tau = \gamma/t \iff t^{-1} = \gamma^{-1}\tau^{-1}$. Therefore, the transition rate $\nu_{AB}$ in the units $t^{-1}$ (for instance, in ps$^{-1}$) may be recovered as
\begin{align}
\nu_{AB} \approx \gamma^{-1}\nu'_{AB}.    
\end{align}
However, the reaction rate is not the only place where the friction coefficient must be accounted for, because the diffusion tensor $M$ is calculated in practice also with a trajectory $X_t$ indexed by time. This is due to the following reason: given $z$ the string method \cite{maragliano2006string} for computing the diffusion matrix $M(z)$ modifies the potential $V$ with a harmonic restraint with stiffness $\kappa$:
\begin{align}
    V^{z}(x) := V(x) + \frac{\kappa}{2}\|z-\xi(x)\|_{2}^{2}. \label{eq: modified potential for diffusion tensor}
\end{align}
With the potential $V^z$, the process $X^z_{t}$ is simulated according to the Langevin dynamics \eqref{eq: full Langevin 1}-\eqref{eq: full Langevin 2}. Then the diffusion tensor is given by:
\begin{align}
    [M(z)]_{jk} &:= \lim_{T \to \infty}\frac{1}{T}\int_{0}^{T}\sum_{i=1}^{3N}m^{-1}_{i}\frac{\partial\xi_{j}(X^{z}_{t})}{\partial x_i}\frac{\partial\xi_{k}(X^{z}_{t})}{\partial x_i}\,dt \\
    &= \lim_{T \to \infty}\frac{1}{T}\int_{0}^{T}\sum_{i=1}^{3N}m^{-1}_{i}\frac{\partial\xi_{j}(X^{z}_{\tau})}{\partial x_i}\frac{\partial\xi_{k}(X^{z}_{\tau})}{\partial x_i}\,\gamma\,d\tau \\
    &\approx \frac{\gamma}{n}\sum_{\ell=1}^{n}\sum_{i=1}^{3N}m^{-1}_{i}\frac{\partial\xi_{j}(X^{z}_{\ell \Delta \tau})}{\partial x_i}\frac{\partial\xi_{k}(X^{z}_{\ell \Delta \tau})}{\partial x_i} \\
    &= \frac{\gamma}{n}\sum_{\ell=1}^{n}\sum_{i=1}^{3N}m^{-1}_{i}\frac{\partial\xi_{j}(X^{z}_{\ell \gamma^{-1}\Delta t})}{\partial x_i}\frac{\partial\xi_{k}(X^{z}_{\ell \gamma^{-1}\Delta t})}{\partial x_i}
\end{align}
The third line following the second is the essence of the string method. Here $\Delta t$ is the time step used for simulating Langevin dynamics--the above calculation suggests that it needs to be additionally smaller to scale with the friction coefficient. 

\section{Relevant differential geometry}
\label{app: B}
Here we introduce some terms from Riemannian geometry used in Section \ref{sec: resmanlearn}. 

\begin{definition}
\label{def: groups}
    We define the Euclidean and standard Euclidean groups in three dimensions as follows: 
\begin{align}
    E(3) &:= \left\{g: \mathbb{R}^{3} \to \mathbb{R}^{3} \mid g(x) = Ox + b, O^{\top}O = I, b \in \mathbb{R}^{3} \right\}, \label{eq: E(3)}\\
    SE(3) &:= \left\{g: \mathbb{R}^{3} \to \mathbb{R}^{3} \mid g(x) = Rx + b, R^{\top}R = I, \textsf{det}(R) = 1, b \in \mathbb{R}^{3} \right\}. \label{eq: SE(3)}
\end{align}
\end{definition}

Thus, E(3) comprises all translations, rotations, and reflections of $\mathbb{R}^{3}$ while SE(3) comprises all rotations and translations. 

\begin{definition}
\label{def: differential}
Let $F : \mathcal{M} \to \mathcal{N}$ be a smooth map between smooth manifolds. The differential (or pushforward) of $F$ at a point $p \in \mathcal{M}$ is the linear map  
\begin{align}
dF_p : T_p \mathcal{M} \to T_{F(p)} \mathcal{N}
\end{align}  
defined by its action on tangent vectors $X \in T_p \mathcal{M}$, satisfying  
\begin{align}
(dF_p X)(f) = X(f \circ F) \quad \text{for all } f \in C^\infty(\mathcal{N}).
\end{align}    
\end{definition}

\begin{definition}
\label{def: immersion}
    A smooth map $F : \mathcal{M} \to \mathcal{N}$ is an immersion at a point $p \in \mathcal{M}$ if its differential is injective:  
\begin{align}
\ker(dF_p) = \{0\} \quad \Leftrightarrow \quad dF_p \text{ is injective}.
\end{align}  
If this holds for all $p \in \mathcal{M}$, then $F$ is called an immersion. Additionally, if $F$ is a homeomorphism onto its image, then it is an embedding and $\mathcal{M}$ is said to be \emph{embedded} in $\mathcal{N}$.  
\end{definition}

\begin{definition}
\label{def: hypersurface}
Let $\mathcal{N}^n$ be an $n$-dimensional smooth manifold. A hypersurface in $\mathcal{N}$ is a smooth manifold $\mathcal{M}^{n-1}$ of dimension $n-1$, together with an embedding  
\begin{align}
F : \mathcal{M}^{n-1} \hookrightarrow \mathcal{N}^n
\end{align}  
such that the image $F(\mathcal{M}) \subset \mathcal{N}$ has codimension 1, meaning:  
\begin{align}
\dim \mathcal{N} - \dim F(\mathcal{M}) = 1.
\end{align}    
Equivalently, a hypersurface can locally be described as the regular level set of a smooth function, i.e there exists an open set $V \subset \mathcal{N}$ and a smooth function $\varphi: V \to \mathbb{R}$ such that  
\begin{align}
\mathcal{M} \cap V= \{ p \in \mathcal{N} : \varphi(p) = 0 \}, \quad \text{with } d\varphi_p \neq 0 \text{ on } \mathcal{M} \cap V.
\end{align}
\end{definition}

\begin{definition}
\label{def: normal vector}
Let $(\mathcal{N}^n, g), (\mathcal{M}^{n-1}, h)$ be a Riemannian manifolds, and let  
\begin{align}
F : \mathcal{M}^{n-1} \hookrightarrow \mathcal{N}^n
\end{align}  
be a smooth embedded hypersurface. A normal vector at a point $p \in \mathcal{M}$ is a nonzero vector $\nu_p \in T_{F(p)} \mathcal{N}$ such that  
\begin{align}
g_{F(p)}(\nu_p, dF_p(X)) = 0 \quad \text{for all } X \in T_p \mathcal{M}.
\end{align}

Since $\mathcal{M}$ is of codimension 1, the space of normal vectors is one-dimensional, and any two normal vectors differ by a scalar multiple.    
\end{definition}

\section{Computational details}

\label{app: Computational details}
The code for reproducing our results and accessing our trained models has been provided at \url{https://github.com/ShashankSule/CV_learning_butane}. Below we include details for the numerical computation of the committor function. 

\textbf{Results in Section \ref{sec: effective dynamics analysis}}. The transition rates in Table \eqref{tab: RatesTheta} were computed by numerically solving the boundary value problem \eqref{eq: committor pde in CVs} and then using a quadrature method for computing the integral \eqref{eq: transition rate in CVs}. For $\xi_1$, the dihedral angle, we discretized the elliptic operator $\mathcal{L}$ with a Fourier difference stencil on $10^3$ equispaced points in $[0, 2\pi)$. After computing the committor by imposing the relevant boundary conditions, we computed \eqref{eq: transition rate in CVs} using Simpson's rule. For $\xi_2$, i.e cosine of the dihedral angle, the committor problem amounts to solving the BVP \eqref{eq: committor pde in CVs} in the interval $(\cos(\pi - 0.2), \cos(\pi/3 + 0.1))$ with boundary conditions $q(\cos(\pi - 0.2)) = 0, \, q(\cos(\pi/3 + 0.1)) = 1$. To numerically solve this BVP, we used Chebyshev spectral differentiation, followed by Clenshaw-Curtis quadrature. Finally, for $\xi_3$, we used a diffusion map-based solver for the committor function proposed for PDEs on manifolds in \cite{evans2022computing, sule2023sharp}. We then used Monte Carlo integration to compute the rate.

\textbf{Neural network architectures and training.} 
For each choice of the feature map, we train a diffusion net carrying the data to a surrogate space, then a confining potential vanishing on the surrogate manifold, and finally collective variable(s) which flow orthogonally to the gradients of the confining potential. Below, we enumerate the network architectures, optimizers, and hyperparameters used in training the models. 

\begin{table}[h!]
\begin{center}
% \scriptsize
 \footnotesize
% \small
\renewcommand{\arraystretch}{3} 
\begin{tabular}{p{2cm} p{2cm} p{2.5cm} p{2cm} p{2cm} p{2cm} p{2cm}}
\toprule
\textbf{Feature map} & \textbf{Model} & \textbf{\# neurons} & \textbf{Activation} & \textbf{Optimizer \& lr} & \textbf{Training epochs} & \textbf{Hyper-parameters}\\
\hline
 \multirow{4}{*}{\shortstack{BondAlign \\ (2,3)}} & LAPCAE & $[12,32,32,32,4]$ & Tanh & Adam, 1e-4 & 500 &  {\shortstack{$\alpha^{Enc}_{LAPCAE} = 0.5$, \\ $\alpha_{LAPCAE} = 2.0$}}\\
\cline{2-7}
 & $\widehat{\Phi}$ & $[3, 30, 45, 32, 32, 1]$ & $x + \sin^2(x)$ & Adam, 1e-3 & 1000 & \shortstack[l]{$\alpha_{Zero} = 1.0$ \\ $\alpha_{normals} = 0.0$} \\
\cline{2-7}
 &  $\widehat{\xi}_1$ & $\widehat{\xi}_1 = \partial_1 \widehat{\Phi}$ & $1 + 2\sin(x)\cos(x)$ & No training & N.A. & N.A. \\
\cline{2-7}
 &  $\widehat{\xi}_2$& $[3, 30, 45, 32, 32, 1]$ & $x^2 + \sin(x)$ & Adam, 1e-2 & 500 & $\alpha_{OC} = 1.0$ \\
 \hline
  \multirow{2}{*}{PlaneAlign} & DNet & $[42,32,32,32,4]$ & Tanh & Adam, 1e-4 & 500 &  $\alpha_{DNet} = 1.0$\\
\cline{2-7}
 & $\widehat{\Phi}$ & $[2, 30, 45, 32, 32, 1]$ & $x + \sin^2(x)$ & Adam, 1e-3 & 1000 & \shortstack[l]{$\alpha_{Zero} = 1.0$ \\ $\alpha_{normals} = 0.0$} \\
% \cline{2-7}
 \hline 
\end{tabular}

\caption{Architectures, optimizers, and hyperparameters for the neural networks trained in the paper.}
\label{tab: NN_details}
\end{center}
\end{table}

\textbf{The Free energy and the diffusion tensor with {\sf BondAlign(2,3)} and {\sf PlaneAlign}.}
In section \eqref{sec: learning cvs} and table \eqref{tab: transition rates} we study transition rates for the anti-gauche transition from CVs learned via the feature maps {\sf BondAlign(2,3)} and {\sf PlaneAlign}. In Figures \eqref{fig: BondAlign23} and \eqref{fig:PlaneAlignF&M} we present the free energy landscapes and diffusion tensors for CVs obtained via algorithm \eqref{alg: learning cv} from data post-processed via these two feature maps. 

\begin{figure}
    \centering
    \includegraphics[width=0.5\linewidth]{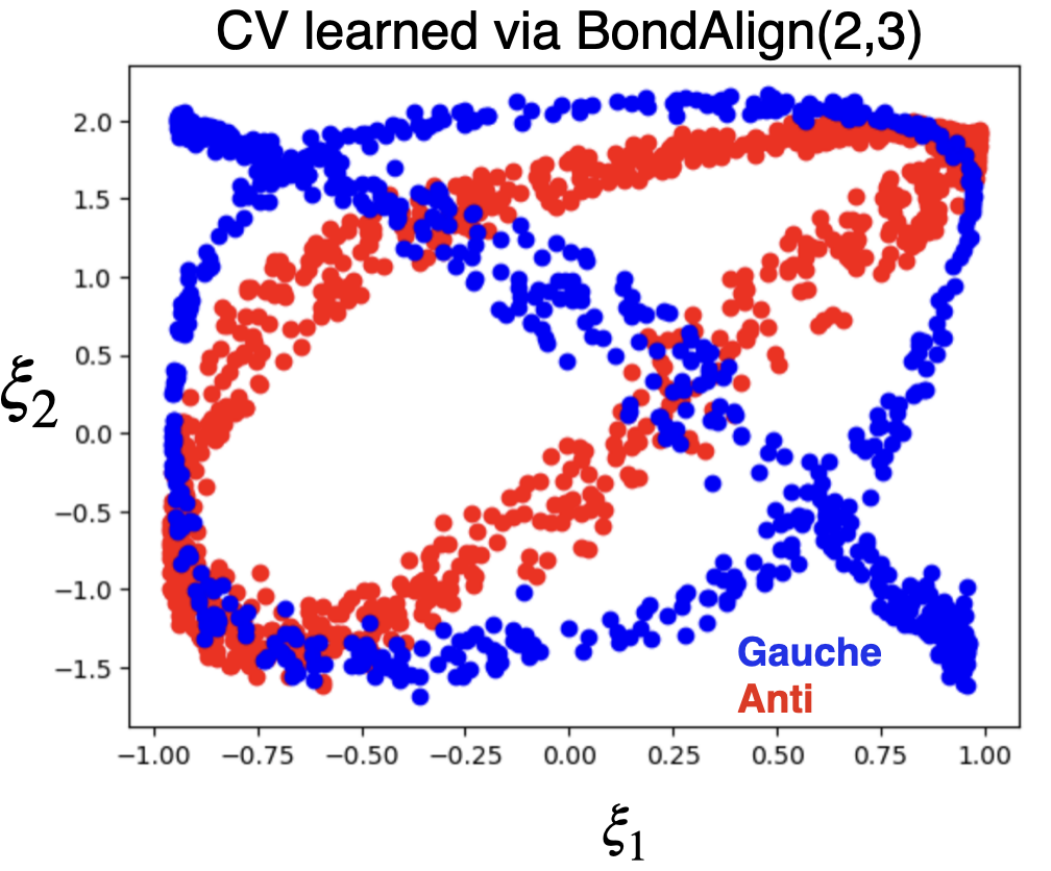}
    \caption{When visualizing data in the CV space learned with the {\sf BondAlign(2,3)} feature map, the anti (red) and gauche (blue) states are not separated.}
    \label{fig: bondalign 23 statesep}
\end{figure}

\begin{figure}[h!]
    \centering
    \includegraphics[width=0.8\textwidth]{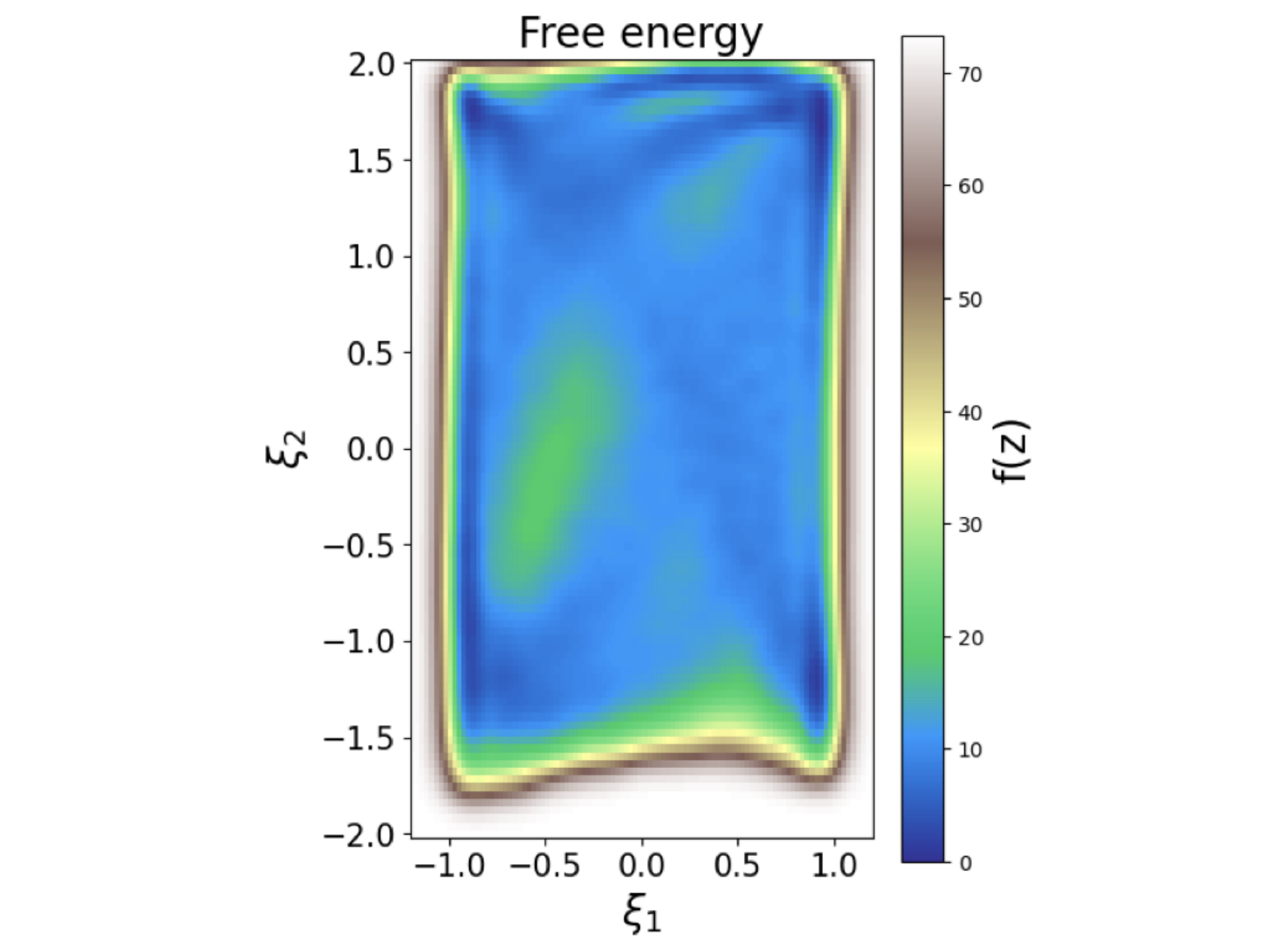}
    \includegraphics[width=0.8\textwidth]{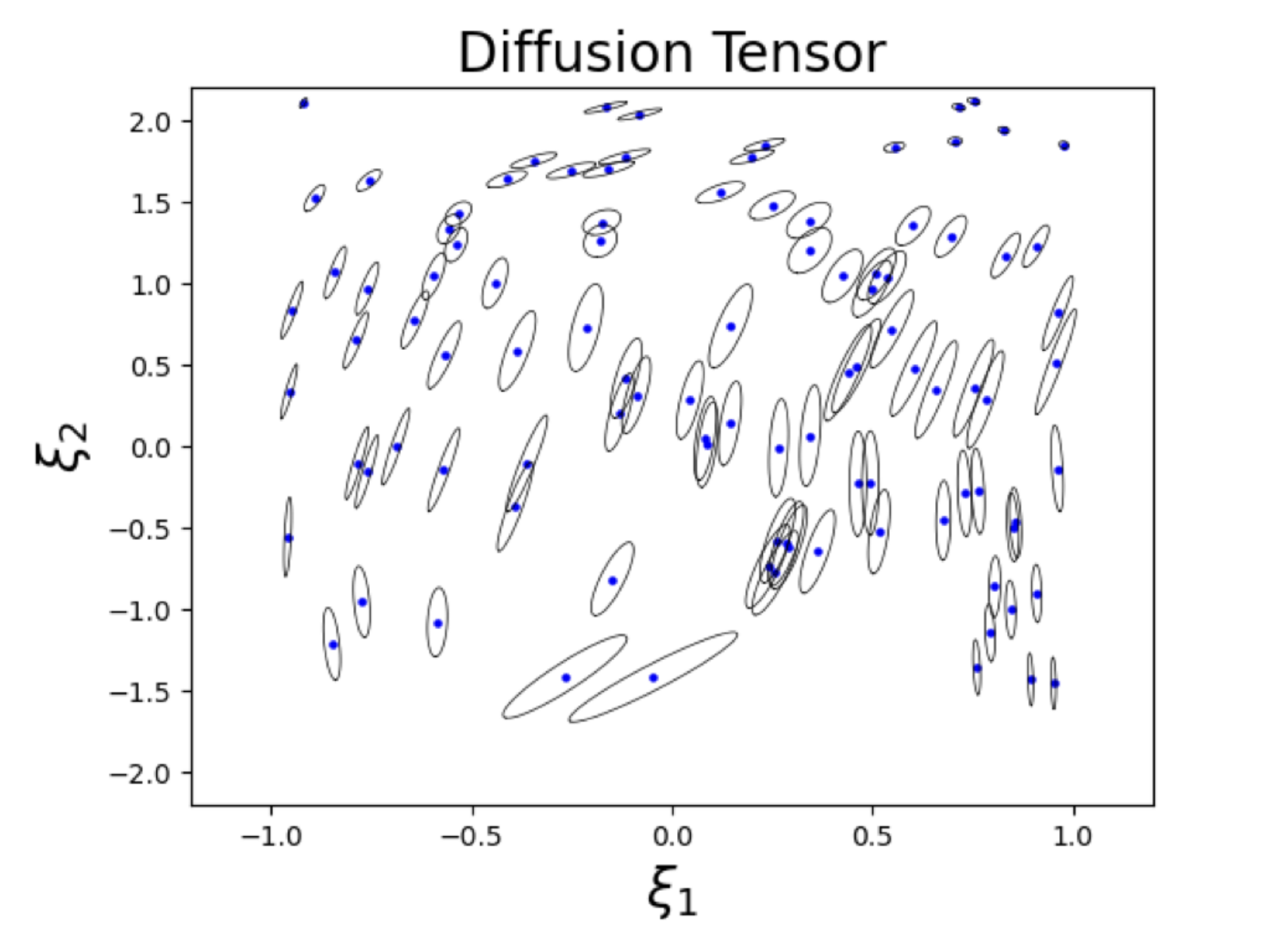}
    \caption{Free energy and diffusion tensor for {\sf BondAlign(2,3)} with LAPCAE as the global coordinate chart parameterizing the resident manifold.}
    \label{fig: BondAlign23}
\end{figure}

\begin{figure}[h!]
    \centering
    \includegraphics[width=0.8\textwidth]{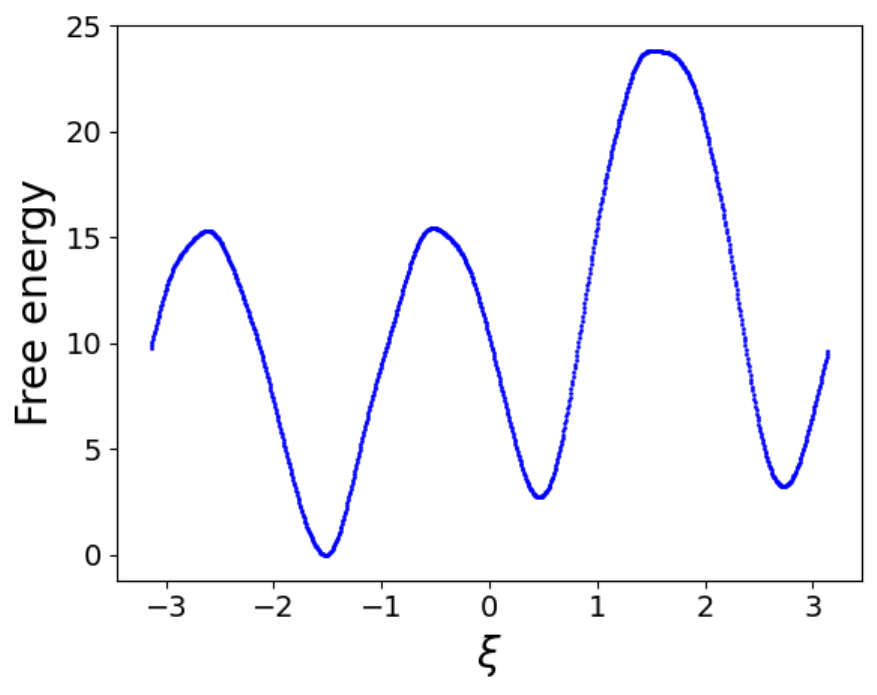} \includegraphics[width=0.8\textwidth]{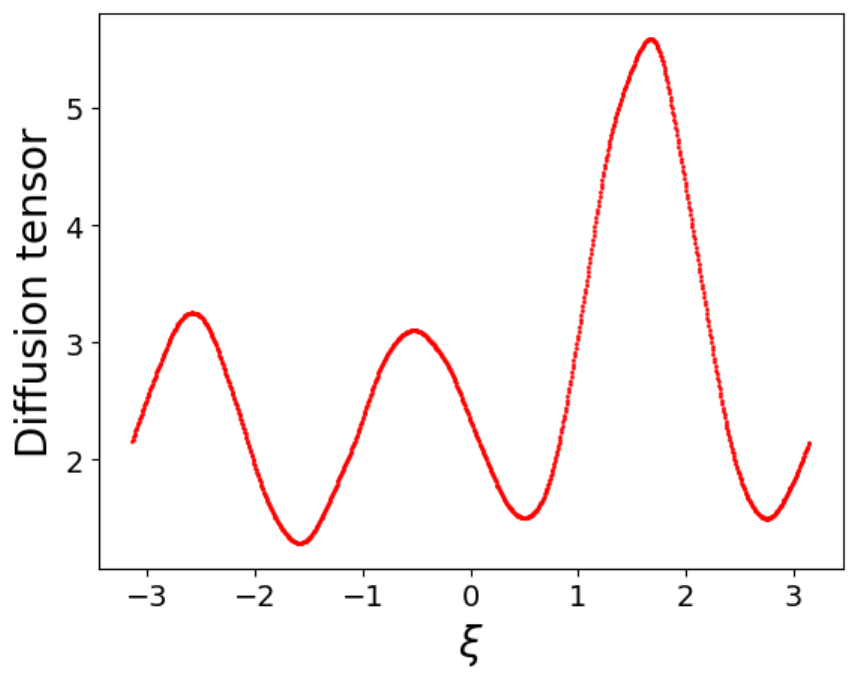}
    \caption{Free energy and diffusion tensor for {\sf PlaneAlign}.}
    \label{fig:PlaneAlignF&M}
\end{figure}

\newpage
\printbibliography
\end{document}